\documentclass[12pt,oneside]{amsart}
\usepackage{amscd,amsmath,amssymb,amsfonts}
\usepackage[all]{xy}
\usepackage{hyperref}
\usepackage{url}
\usepackage{stmaryrd}

\usepackage[a4paper, total={6in, 8in}]{geometry}

\usepackage{xcolor}

\usepackage{upgreek}

\theoremstyle{plain}
\newtheorem{thm}{Theorem}
\newtheorem{lem}[thm]{Lemma}
\newtheorem{cor}[thm]{Corollary}
\newtheorem{prop}[thm]{Proposition}

\newtheorem{conj}[thm]{Conjecture}
\theoremstyle{definition}
\newtheorem{ex}[thm]{Example}
\newtheorem{defn}[thm]{Definition}

\newtheorem{claim}[thm]{Claim}
\newtheorem{rmk}[thm]{Remark}

\numberwithin{thm}{section}
\numberwithin{equation}{section}

\newcommand{\eq}[2]{\begin{equation}\label{#1}#2 \end{equation}}
\newcommand{\ml}[2]{\begin{multline}\label{#1}#2 \end{multline}}
\newcommand{\ga}[2]{\begin{gather}\label{#1}#2 \end{gather}}
\newcommand{\gr}{{\rm gr}}
\newcommand{\fil}{{\rm fil}}

\newcommand{\im}{{\rm im}}
\newcommand{\Spec}{{\rm Spec \,}}

\newcommand{\sC}{{\mathsf C}}

\newcommand{\sF}{{\mathsf F}}
\newcommand{\sG}{{\mathsf G}}
\newcommand{\sH}{{\mathcal H}}

\newcommand{\sL}{{\mathsf L}}

\newcommand{\sO}{{\mathcal O}}

\newcommand{\sRZ}{{\mathsf{RZ}}}
\newcommand{\A}{{\mathbb A}}

\newcommand{\F}{{\mathbb F}}

\newcommand{\N}{{\mathbb N}}
\renewcommand{\P}{{\mathbb P}}
\newcommand{\Q}{{\mathbb Q}}

\newcommand{\Z}{{\mathbb Z}}

\newcommand{\hh}{{\rm h}}
\newcommand{\Iw}{{\rm Iw}}
\newcommand{\RZ}{{\rm RZ}}
\newcommand{\nil}{{\rm nil}}
\newcommand{\per}{{{}^{\mathrm p}\!}}

\newcommand{\QQl}{\overline{\mathbb Q}_\ell}

\begin{document}

\title{Semistable Lefschetz Pencils}
\author{H\'el\`ene Esnault \and Moritz Kerz}
\address{Freie Universit\"at Berlin, Berlin,  Germany;
{Harvard} University,  Cambridge, USA
}
\email{esnault@math.fu-berlin.de}
\address{   Fakult\"at f\"ur Mathematik \\
Universit\"at Regensburg \\
93040 Regensburg, Germany}
\email{moritz.kerz@mathematik.uni-regensburg.de}
\thanks{The second author is supported by the SFB 1085 Higher Invariants, Universit\"at Regensburg.}

\begin{abstract}
We study the geometry and cohomology of  Lefschetz pencils for semistable schemes  over a
discrete valuation ring.  We relate the global cohomological properties of the Lefschetz pencil
and the monodromy-weight conjecture, in particular we show that if one assumes the
monodromy-weight conjecture in smaller dimensions  then one can obtain a rather complete
understanding of the relative cohomology of the pencil. This reduces the monodromy-weight
conjecture to an arithmetic variant of a conjecture of Kashiwara for the projective line.
\end{abstract}

\maketitle

\section{Introduction}

\subsection{Background}

The theory of Lefschetz pencils is an important tool in the study of the topology of algebraic varieties
which originates in the work of Picard and Lefschetz~\cite{Lef24}. The theory has been  extended to
positive characteristic and \'etale cohomology by Deligne and Katz~\cite{SGA7.2}. The idea of Lefschetz is to fiber a smooth
projective variety $X$ of dimension $n$ over a field $k$ after a  certain blow-up $\tilde X\to X$ into a pencil $\phi\colon
\tilde X\to \P^1_k$ such that $\phi$ is smooth except for a finite number of quadratic
singular points  located in different fibers, i.e.\ $\phi$ is like a real Morse function in differential topology. The
Leray spectral sequence for the pencil map $\phi$ allows one  to study the cohomology of $X$.
This topological idea was a key ingredient in Deligne's   first proof of the Weil
conjectures~\cite{Del74}.

The geometric theory of Lefschetz pencils was generalized by Jannsen--Saito~\cite{JS12} to smooth
projective schemes over a discrete valuation ring. In this note we study the geometry and
cohomology of Lefschetz pencils of (strict) semistable schemes  over discrete valuation rings.

One motivation for this study is the monodromy-weight conjecture \cite[Section~8.5]{Del70},  \cite[p.23]{RZ82},
which is an analog of the Riemann hypothesis part of the Weil conjectures over a $p$-adic
local field. In~\cite{RZ82}, Rapoport--Zink studied the monodromy-weight conjecture in terms of a
monodromy-weight spectral sequence of a semistable model, which they define. The key unsolved problem is
that the $d_1$-differential  in their spectral sequence is not well-understood in relation
to the monodromy operator $N$. It turns
out that, assuming the monodromy-weight conjecture is known in smaller dimensions,  the
analog of this $d_1$-differential in the  relative cohomology of the semistable Lefschetz pencil is easy to
control. This allows us to reduce the monodromy-weight conjecture to an unsolved problem about the
cohomology of a Picard-Lefschetz sheaf on $\mathbb P^1$, see~\ref{ss:mct}.

\subsection{Geometry of semistable Lefschetz pencils}

Let $\sO$ be a henselian  discrete valuation  ring with perfect residue field
$k$  of characteristic
$\mathrm{ch}(k) \ne 2$.
Define  $K$ to be the field of fractions $K=
\mathrm{frac}(\sO)$, which we assume to be of characteristic zero.
Let $X\subset \P^N_\sO$ be a (strict) semistable projective scheme over
$\sO$, i.e.\ its special fiber $X_k$ is a simple normal  crossings divisor. We endow  $X$
with the usual stratification, see the beginning of Section~\ref{subsec:semistabmor}. A sufficiently
general linear subspace of codimension two  $A\subset \P^N_\sO$  gives rise to a pencil map
\[
\phi\colon \tilde X \to \P^1_\sO,
\]
where $\tilde X= \mathrm{Bl}_{A\cap X}(X)$, see Section~\ref{sec:geolefschetz}.

A point $x\in \tilde X$ is called {\em critical} if either $x$ lives over $K$ and is a non-smooth point of $\phi_K$
or if $x$ lives over $k$ and $\phi|_Z\colon Z\to \P^1_k$ is non-smooth at $x$, where
$Z\subset \tilde X_k$ is the
stratum containing $x$.

Roughly speaking, we say that the pencil $\phi_k\colon \tilde X_k\to \P^1_k  $ is a {\it stratified
Lefschetz pencil} over $k$ if $\phi_k\colon \tilde X_k\to \P^1_k$ is a stratified Morse
function in the usual sense, Definition~\ref{dfn:nondegcp}, with at most one critical
point per geometric fiber; for the precise definition of stratified Lefschetz pencil see
Definition~\ref{def:stratLefp}.
It is not hard to show that after a suitable Veronese embedding a generic
choice of $A_k$ gives rise to a stratified Lefschetz pencil $\phi_k$.

The following theorem has been observed in  \cite[Theorem~1]{JS12}  in the smooth case.

\begin{thm}[see Theorem~\ref{thmmain:lefpen}] Assume that $\mathrm{ch}(k)> \dim(X_K)+1 $
  or $\mathrm{ch}(k)=0$. Then the following properties hold.
\begin{itemize}
 \item[(1)]  If $\phi_k$ is a stratified Lefschetz pencil then $\phi_K$ is a Lefschetz pencil.
 \item[(2)] The subset $S$ of critical points in $X$ is closed and with the reduced subscheme
  structure it maps isomorphically onto its image in $\P^1_\sO$.
  \item[(3)]  Each connected  component of $S$ is a trait (i.e.\ the spectrum of a henselian
  discrete valuation ring) which is finite and of ramification
  index over $\sO$ equal to the number of irreducible components of $X_k$ it meets.
  \end{itemize}
\end{thm}

For
$\mathrm{ch}(k)\ne 2$ one only
 obtains a slightly weaker result which we formulate in
Theorem~\ref{thm:mainpropermorse} for a general stratified Morse function.

\subsection{Cohomology of semistable Lefschetz pencils}

Let $\ell$ be  a  prime number  invertible in $\sO$ and assume that the residue field $k$
is finite. Set $\Lambda= \overline \Q_\ell$ and $n=\dim(X_K)$.
Let $\phi\colon \tilde X\to \P^1_\sO$ be  a semistable Lefschetz pencil, i.e.\ $\phi_k$
is a stratified Lefschetz pencil and $\phi_K$ is a Lefschetz pencil.

We endow $X$ with the
middle perversity, see~\ref{app:subsecperv}, so that $\Lambda[n+1]\in D^b_c(X_K,\Lambda)$
is a perverse sheaf.   In order to understand the relative cohomology of a semistable
Lefschetz pencil one has to study the degeneration of the classical Picard-Lefschetz perverse sheaf $\sL_K = \per R^0
\phi_{K,*} (\Lambda[n+1])$ in terms of its nearby cycle perverse sheaf $\sL_k = R\Uppsi_{X/\sO}(\sL_K)[-1]$.   Note that for $i\ne 0$ the perverse sheaf $ \per R^i
\phi_{K,*} (\Lambda[n+1])$ is geometrically constant, and it can be analyzed by the
Weak Lefschetz theorem~\cite[Th\'eor\`eme 4.1.1]{BBD83}  by cutting with a   hyperplane section.

As a consequence of results of Grothendieck~\cite[Expos\'e I]{SGA7.1} and Rapoport--Zink\ \cite{RZ82} the  monodromy action on $\sL_k$ is
unipotent, see Proposition~\ref{prop:grorzuni} and Lemma~\ref{lem:unipopl}, in particular it is given in terms of a nilpotent
operator $N\colon \sL_k\to \sL_k(-1)$. This operator induces a monodromy filtration $\fil^{\rm
  M}$ on the perverse sheaf $\sL_k$, see~\eqref{eq:convmofil}.

In order to understand the structure of $\sL_k$ we have to assume that the monodromy-weight
conjecture
holds in dimensions smaller  than $\dim(X_K)$, see~\ref{ss:mct}. Recall that
the monodromy-weight  conjecture is known for dimension at most two \cite[Satz~2.13]{RZ82}.

\begin{thm}[see Theorem~\ref{thm:mainmwind}] \label{thm:main}  Assume that the monodromy-weight conjecture is known over $K$ for dimensions
  smaller than $n=\dim(X_K)$. Then  the following properties hold.
  \begin{itemize}
  \item[(1)] For any $a\in \Z$ the monodromy graded piece $\gr^{\rm M}_a \sL_k$ is pure of weight
  $n+a$, in particular it is geometrically semisimple (\cite[Corollaire~5.4.6]{BBD83}).
  \item[(2)] The non-constant part of  $\gr^{\rm M}_a \sL_{\overline k}$ satisfies multiplicity one.
  \end{itemize}
\end{thm}

In Corollary~\ref{cor:mainredmw} we show that in order to prove the monodromy-weight conjecture by
induction on $n=\dim(X_K)$, one would have to prove that the monodromy filtration of
$H^0(\P^1_{\overline k}, \sL_{\overline k})$ agrees with the filtration induced by the
spectral sequence of the filtered perverse sheaf $(\sL_{\overline k},\fil^{\rm M})$. We call this the
{\em monodromy property}, see Definition~\ref{def:mwproperty}. The monodromy property is
of ``purely topological''
nature and is known to hold in our setting for $\sO$ of equal characteristic. It plays an
important role in the theory of Hodge modules~\cite{Sai88}  and   twistor   $\mathcal D$-modules~\cite{Moc07}.
In mixed characteristic it fits into what we like to call the arithmetic Kashiwara
conjecture, see Conjecture~\ref{conj:arkash}, motivated by the Kashiwara conjecture in complex geometry \cite{Kas98}. We defer the study to a   forthcoming work.

Because the monodromy action on $\sL_k$ is unipotent,   it is  tame. This tameness
generalizes to all models of $\P^1_K$ in the following sense.

\begin{thm}[see Corollary~\ref{cor:tamecor}]\label{thmint:3}
Assume $\mathrm{ch}(k)>n+1$. For any closed point $x\in \P^1_K$ and any $i\in \Z$,  the $\mathrm{Gal}(\overline x/x)$-action on
$H^i(\phi^{-1}(\overline x),\Lambda)$ is tame.
\end{thm}

In fact, we prove a slightly weaker result for $\mathrm{ch}(k)\ne 2$ in Theorem~\ref{tame:tamethm}.

For $x$ specializing to a regular value in $\P^1_k,$ Theorem~\ref{thmint:3} is an immediate
consequence of the aforementioned results  of Rapoport--Zink  about semistable reduction.
The proof of Theorem~\ref{thmint:3} uses the Grothendieck-Murre criterion of tameness and
Nakayama's generalization  \cite[Theorem~0.1]{Nak98}  of  the work of  Rapoport--Zink
  to the log-smooth case applied to a
certain blow-up. It also relies on the tameness statement in the Picard-Lefschetz formula,
see Proposition~\ref{prop:picardlefschetzform}.

In forthcoming work we will study the monodromy property in terms of a tilting to equal
characteristic zero, for which  the tameness as formulated in Theorem~\ref{thmint:3} plays
a role.

\subsection{Content}

As a technical ingredient, which is not well-documented
in the literature and which is necessary for our study of semistable Lefschetz pencils,  we develop the algebraic theory of stratified Morse functions for
special algebraic analogs of Whitney stratifications, which we call regular
stratifications, in Sections~\ref{sec:1} and~\ref{sec:regstrat}.

In Section~\ref{subsec:semistabmor} we prove the main theorem about Morse functions for
semistable schemes.
 In Section~\ref{sec:geolefschetz} we generalize without difficulty the usual geometric theory of Lefschetz pencils to the
 stratified and semistable context.

In Section~\ref{sec:coholefcl} we recast the classical cohomology theory of Lefschetz
pencils in terms of perverse sheaves. The only  new result is the generalization of a
central observation of Katz~\cite[Expos\'e~XVIII, Th\'eor\`eme 5.7]{SGA7.2}   for type (A) pencils to the general case.

Section~\ref{subsec:mofil} summarizes properties of the monodromy filtration in the
context of perverse sheaves.

Our presentation of the theory of Rapoport--Zink in Section~\ref{sec:rapzi} is novel in that we make
full use of the duality theory of the nearby cycle functor. This  allows us to give a
clear-cut axiomatic description of  their  construction,
and  of its perverse
formulation in~\cite[Section~2.2]{Sai03}. In order to give a  coordinate-free presentation in the  case of
$\Z/\ell^\nu \Z$-coefficients, we use Beilinson's Iwasawa twist. As this formalism, which
is extremely useful for unipotent nearby cycles, is not
well-documented in the literature, we summarize it in Appendix~\ref{sec:appendix}.

Our main cohomological results about semistable Lefschetz pencils and the relation to the
monodromy-weight conjecture are studied in Section~\ref{sec:cohosemistlef}.

Section~\ref{sec:tamepl} discusses tameness of the Picard-Lefschetz sheaf.

\subsection*{Notation} By  $k$ we denote  a  field of characteristic $\ne 2$, which is assumed to be
perfect if not stated otherwise.
By $\sO$ we denote a henselian discrete valuation ring
with  residue field $k$. We always assume that $K=\mathrm{frac}(\sO)$ has
characteristic $0$.
 When mentioning the dimension  of a scheme, we assume that
it is equidimensional.
By $\ell$ we denote a prime number invertible in $k$.

\subsection*{Acknowledgement} We thank Alexander Beilinson for the discussions we had
since July 2021 on and around the topic of this note. In particular, our understanding of
vanishing cycles profited substantially from his many explanations. This note is written
using in part his language.
We thank Takuro Mochizuki for explaining to us the scope of Kashiwara's conjecture and for
communicating
some counterexamples.  We thank the referee for their precise work which helped us to
improve the exposition of the article.

Our note was finalized at the Department of Mathematics of Harvard
University, which we thank for the excellent working conditions.

\section{Morse morphisms and Morse functions}\label{sec:1}

\subsection{Morse morphisms}

 Let $X$ be a   scheme locally of finite type over  a field
$F$ with $\mathrm{ch}(F)\ne 2$.
For $F$ algebraically closed we say that $X\to \Spec F$ is a {\em Morse  morphism} or
simply  that
$X$  {\it is Morse } if for any closed point $x\in X$, either $X$ is regular at $x$
or there is an isomorphism of
$F$-algebras
\[
 \sO_{X,x}^\hh\cong
F [ X_0,\ldots, X_n]^\hh /(X_0^2 + \ldots + X_n^2)
\]
for some $n\ge 0$. Here  $^\hh$ denotes  the  henselization
with respect to the maximal ideal
$(X_0 , \ldots , X_n )$.
Over a general field $F$  we say that $X\to \Spec F$ is a
{\em Morse morphism}  if $X\otimes_F \overline F\to \Spec \overline F$ is a Morse morphism, where $\overline F/F$
is an algebraic closure. A non-smooth   point of a Morse morphisms $X\to \Spec F$ is called a {\em
  non-degenerate quadratic singularity}.

Let now $X$ and $Y$ be schemes on which $2$ is invertible.
A morphism of schemes $\phi\colon X\to Y$ is called a  {\em Morse morphism}  if it is locally
finitely presented, flat and for
any point $y\in Y$ the fiber $X_y$ is Morse over $k(y)$.
Clearly, Morse morphisms are  preserved by base change. Observe that in
 \cite[Section~(3.6)]{Del80},  it is called  ``essentiellement lisse'', in \cite[Section~4]{JS12}  the terminology
  ``almost good reduction'' is used.

We recall the deformation theory of non-degenerate quadratic singularities~\cite[Expos\'e~XV, Th\'eor\`eme~1.1.4]{SGA7.2}.

\begin{prop}\label{prop1:struc_almsm}
Let $\phi\colon X\to Y$ be flat and locally of finite presentation such that $\phi^{-1}(y)$ is
Morse over $k(y)$ for a $y\in Y$ with $k(y)$ separably closed. Let $x\in \phi^{-1}(y)$ be a singular point of
$\phi^{-1}(y)$. Then there exists  an isomorphism of $\sO_{Y,y}^{\rm h}$-algebras
\[
\sO_{X,x}^{\rm h}\cong \sO_{Y,y}[X_0,\ldots , X_n]^{\rm h}/(X_0^2 + \ldots + X_n^2-\alpha),
\]
where $\alpha$ is in the maximal ideal $\mathfrak m_y\subset\sO_{Y,y}^{\rm h}$.
\end{prop}

Here on the right the  henselization is with respect to the maximal ideal generated by
$\mathfrak m_y$  and $ (X_0,\ldots  , X_n)$.

\subsection{Morse functions}

Let in the following $X$  be a regular noetherian scheme  and let $D$ be a
one-dimensional, regular, noetherian scheme. We always assume that $2$ is
invertible on $X$ and on $D$.
  Let $\phi\colon X\to D$  be a morphism of finite type. A non-smooth point $x\in X$ of $\phi$ is
called a {\em critical point} of $\phi$, the image $\phi(x)$ of a critical point $x\in X$ is called a
{\em critical value} of $\phi$.  In particular the closed points $x$ with ${\rm dim}_xX=0$ are critical.
 We say that $\phi\colon X\to D$ is a {\em
  Morse function} if $\phi$ is a Morse morphism around any point $x\in X$ with $\dim_x(X)>0$.
We call a critical point $x\in X$ {\em non-degenerate} if $\phi$ is
a Morse function in a neighborhood of $x$.

Note that in Morse theory one sometimes asks that additionally a Morse function has at most one critical
point per fiber.

\begin{rmk}\label{sec:mo:rmkiso}
The set of critical points of a Morse function
$\phi$ consists of finitely many closed points, compare proof of
Corollary~\ref{sec:morse_ml} below.
\end{rmk}

\medskip

For us $k$ will be a perfect field of characteristic different from two. For the rest of
this section we assume that $k$ is algebraically closed, that $X$ and $D$ are of finite
type over $k$ and  that $\phi\colon X \to D$ is a $k$-morphism, $n=\dim(X)$.
Proposition~~\ref{prop1:struc_almsm} has the following corollary.

\begin{cor}[Morse lemma]\label{sec:morse_ml}
For   $x\in
X$ a critical point of a Morse function $\phi\colon X\to D$   with $n\ge 1$  there exist
$k$-isomorphisms
\[
\sO_{X,x}^{\rm h}\cong k[X_1 , \ldots , X_n]^{\rm h},\quad \sO_{D,\phi(x)}^{\rm h}\cong k[T]^{\rm h}
\]
such that $\phi^\hh_x\colon \sO_{D,\phi(x)}^{\rm h} \to \sO_{X,x}^{\rm h}$ maps $T$
to $X_1^2 + \ldots + X_n^2$.

  If $n=0,$ then $\sO_{X,x}^{\rm h}=k$  and $\phi(x)^{\rm h}(T)=0$.
\end{cor}
\noindent
 Here the henselizations $^\hh$ are with respect to the maximal ideals $(X_1,\ldots, X_n)$ and $(T)$.
\begin{proof}
It suffices to show the corollary for $x$ a closed point as then a posteriori it follows that there
are no non-closed critical points, using that the set of critical points is closed.
Looking at the henselian local presentation of $\sO_{X,x}$ as an
$\sO_{D,\phi(x)}$-algebra from Proposition~\ref{prop1:struc_almsm} we see that necessarily
$\alpha\in \sO_{D,\phi(x)}$ is a uniformizer, as otherwise  $\sO_{X,x}$ would be singular.
\end{proof}

\smallskip

Let $J_\phi$ be the Jacobian ideal of $\phi$, see~\cite[Section 4.4]{HS06}.

\begin{lem}\mbox{}
\begin{itemize}
\item[(i)] $J_\phi=\sO_X$ $\Leftrightarrow $ $\phi$ is smooth;
 \item[(ii)] $V(J_\phi)$ is a finite disjoint union of copies of $\Spec k$  $\Leftrightarrow$
  $\phi$ is a Morse function.
  \end{itemize}
\end{lem}

The proof of part (ii) of the lemma is explained in~\cite[Expos\'e XV, Section 1.2]{SGA7.2}.

\section{Regular stratifications}\label{sec:regstrat}

\subsection{Regular stratifications and critical points}
Let $X$  be a noetherian scheme  and let $D$ be a
one-dimensional, regular, noetherian scheme. We always assume that $2$ is
invertible on $X$ and on $D$.

\begin{defn} A {\em stratification}  of $X$ is a finite set $\bf Z$ consisting of disjoint locally
closed subsets $Z\subset X$ called {\em strata} such that  $X$ is the  disjoint union of the strata and such that the Zariski closure
$\overline Z$ of a stratum $Z \in \bf Z$ is a union of
strata.  If $X$ is a scheme and $\bf Z$ a stratification of
$X$ the pair $(X, \bf Z)$ is called a {\it stratified scheme.}
 \end{defn}

For $(X,\mathbf Z)$ a stratified scheme and $x\in X$ we denote by $Z_x\subset X $ the
stratum of $x$.  We usually endow a stratum with its reduced subscheme structure.
What we call a stratification is  called a good stratification in \cite[Definition~09XZ]{StPr}. We call the stratification $\bf Z$ {\em regular} if the closure
$\overline Z$ with the reduced subscheme structure is regular for any stratum $Z\in \bf
Z$.

Note that for a morphism of finite type $f\colon Y\to X$, the  set-theoretic  pullback $f^{-1}({\bf Z})$ of
$\bf Z$ to $Y$ is
in general not a stratification but just a  partition of $Y$ into locally
closed subsets which we call the {\em pullback partition}.
If $f$ is
flat,  then $f^{-1}({\bf Z})$ is automatically a stratification.

\smallskip

Now assume that $X$ is endowed with a regular stratification $\bf Z$. Consider a morphism
$\phi\colon X\to D$ of finite type.

A point $x\in X$ is called a {\em critical point} of $\phi$ (with respect  to the stratification $\bf Z$) if
 $ \phi|_Z\colon Z \to D$  is non-smooth at $x$, where $Z$ is the stratum
containing $x$ endowed with the reduced subscheme structure. Otherwise, $x$ is called {\it non-critical}.
For $x\in X$ a critical point we call $\phi(x)\in D$ a {\em critical value}.

\begin{lem}\label{lem:crtilocclos}
   The set of critical points of $\phi$ is closed in $X$.
\end{lem}
\begin{proof}
  Let $C\subset X$ be the subset of critical points. As for any stratum $Z \in \bf Z$ the
  subset $C\cap Z$ is the non-smooth locus of  $\phi|_{Z}\colon Z \to D$,  we see that
  $C\cap Z\subset Z$ is closed. Therefore, $C$ is constructible in $X$. So by \cite[Lemma~0903 (2)]{StPr} we just
  have to show that $C$ is closed under specialization. So consider points
  $x_2\in \overline{\{ x_1\}}$ with $x_1\in Z_1\in \bf Z$ and $x_2\in Z_2\in \bf Z$.
  Assume $x_1\in C$ and $Z_1\ne Z_2$.

  \medskip

  {\em 1st case}   $Z_2$ is non-flat over $D$ at $x_2$.\\
  Then of course  $Z_2$  is not smooth over $D$ at $x_2$, so $x_2\in C$.

   \medskip

   {\em 2nd case}:  $Z_2 $  is flat over $D$ at $x_2$.\\
   Then also   $\overline Z_1$  is flat over $D$ at $x_2$, because flatness is equivalent to
   being dominant over $D$. As   $Z_2 \hookrightarrow \overline{
   Z}_1 $  is a closed immersion of regular schemes there exists a regular sequence
   $\underline a$ in $\sO_{\overline Z_1, x_2}$ with $ \sO_{\overline Z_1, x_2}/(\underline
   a)=\sO_{Z_2, x_2}$. Let $\pi\in \sO_{D,\phi(x_2)}$ be a uniformizer. As $\pi$ is a
   non-zero divisor on $\sO_{Z_2, x_2}$ we see that $\underline a,\pi$ is also a regular
   sequence in $\sO_{\overline Z_1, x_2}$, so the image of $\underline a$ in
   $\sO_{\overline Z_1, x_2}/(\pi)$ is also a regular sequence, and it remains a regular
   sequence after the base change by the algebraic closure
   $\overline{k(\phi(x_2))}/k(\phi(x_2))$.
   So as $\sO_{\overline Z_1, x_2} \otimes_{\sO_{D,\phi(x_2)}} \overline{k(\phi(x_2))}$ is
   singular, since $x_1\in C$ and since the non-smooth locus of the map $\overline Z_1 \to D$ is
   closed.
Then also   $\sO_{ Z_2, x_2} \otimes_{\sO_{D,\phi(x_2)}} \overline{k(\phi(x_2))}$ is
singular, because it is a quotient of the ring $\sO_{\overline Z_1, x_2}
\otimes_{\sO_{D,\phi(x_2)}}$ modulo a regular sequence.
   We have shown  $x_2\in C$. This finishes the proof.
\end{proof}

The following definition is copied from stratified Morse theory~\cite[2.0]{GM88}.

\begin{defn}[Non-degenerate critical points]  \label{dfn:nondegcp}
A critical point $x$ of $\phi$ is called {\it non-degenerate} if for $Z\subset X$ the
stratum containing $x$ and for any stratum $Z'\ne Z$ with $Z\subset \overline Z'$ the following holds:

\begin{itemize}
\item[(1)] $\phi|_{Z}$ has a non-degenerate critical point at $x$ in the sense of Section~\ref{sec:1};
\item[(2)]  $\phi|_{\overline Z'}$ is smooth at $x$.
\end{itemize}

If every critical point of $\phi$ is non-degenerate we say that $\phi$ is a  {\it stratified Morse function}.
\end{defn}

\begin{lem}
  The set of critical points of a stratified Morse function $\phi$ consists of finitely
   many closed points.
\end{lem}

\begin{proof}
  Any  critical point $x$ of $\phi$ is closed.
Indeed, let $Z$ be the stratum of $x$.
Assume $x$ is not closed in $X$. Then there exists a proper specialization $x'$ of $x$. If
$x'$ lies in $Z$ then this contradicts Remark~\ref{sec:mo:rmkiso} applied to
$\phi|_Z\colon Z\to D$. If $x'$ lies
in a different stratum $Z'$ then $\phi|_{Z'}\colon Z'\to D$ is non-smooth at $x'$ by
Lemma~\ref{lem:crtilocclos}. As also   $\phi|_{\overline Z} \colon \overline Z \to D$ is
non-smooth at $x'$ this would contradict the condition that $\phi$ is  a stratified Morse function around $x'$.

 To conclude use that the set of critical points is
closed by Lemma~\ref{lem:crtilocclos}.
\end{proof}

\subsection{Stratified regular immersions}\label{subsec:stratregim}

Let $(X,\mathbf Z_X)$ and $(Y,\mathbf Z_Y)$ be two noetherian schemes with regular
stratifications. Let $i\colon Y\hookrightarrow X$ be a regular closed immersion of
codimension $c$. We say that $i$ is {\em stratified regular} if for any point $y\in
Y$
after replacing $X$ by an open neighborhood of $i(y)$ and $Y$ by its preimage it holds:
\begin{itemize}
\item[ (1)]  the schematic pullback $i^{-1}(Z_{i(y)})$ is reduced and equal to $Z_y$;
 \item[ (2)]  $i|_{Z_y}\colon Z_y\to Z_{i(y)} $ has  codimension $c$.
\end{itemize}
Here $Z_y$
is the stratum of $y\in Y$ and $Z_{i(y)}$ is the stratum of $i(y)\in X$.

 For a stratified regular closed immersion $i\colon Y\hookrightarrow X$ and
  for a stratum $Z\subset X$,   each connected component  of
  $i^{-1}(Z)$ is a connected component of a stratum  of $Y$.
Note that locally around $x=i(y)$ a regular sequence for $i$  restricts to a
regular sequence for $Z_y\hookrightarrow Z_x$ by~\cite[Theorem 17.4]{Mat86}.

If $Y$ is not endowed with a stratification but $(X,\mathbf Z_X)$ has a regular stratification,  we call
the closed immersion $i\colon Y\hookrightarrow X$ {\it stratified regular} if the partition $i^{-1}( \mathbf
Z_X)$ is a regular stratification and with this stratification $i$ is stratified regular.

\begin{lem}\label{lem:strregpull}
For a stratified regular immersion $i\colon( Y,\mathbf Z_Y)\hookrightarrow (X,\mathbf
Z_X)$  and a stratum $Z\in \mathbf Z_X$ the preimage $i^{-1}(\overline Z)$ is regular.
\end{lem}

\begin{proof}
Consider a point  $y\in i^{-1}(\overline Z)$, $x=i(y)$ and
 a regular sequence
$\underline a\in \sO_{X,x}$  generating the ideal of $i$. Then the image of $\underline
a$ in $\sO_{\overline Z,x}$ is part
of a regular parameter system as its image in $\sO_{Z_{x},x}$ has this property.
\end{proof}

The proof of the following lemma is similar.

\begin{lem}\label{lem:stratregtransv}
 For a regularly stratified scheme $(X,\mathbf Z_X)$ and  for a regular closed immersion
  $i\colon Y\hookrightarrow X$ and a point $y\in Y$ the following are equivalent:
  \begin{itemize}
 \item[ (1)]  $i$ is a stratified regular immersion in a neighborhood of $y$;
 \item[ (2)]
for a regular sequence
$\underline a\in \sO_{X,i(y)}$ locally generating the ideal of $i$,  the sequence
$\underline a|_{Z_{i(y)}}$ is part of a  regular parameter system in $\sO_{Z_{i(y)},i(y)}$.
\end{itemize}
\end{lem}

Let now $k$ be an infinite perfect field and $X\hookrightarrow  \P^N_k$  be an immersion.
Recall the following  Bertini type  theorem.

\begin{prop}\label{prop:bertini} For a regularly stratified scheme $(X,\mathbf Z_X)$  the following properties are verified.
\begin{itemize}
\item[(i)] For  a generic hypersurface  $H\hookrightarrow \P^N_k$ the map
  $i\colon X\cap H \hookrightarrow X$ is a stratified regular immersion.
 \item[ (ii)]  Assume given a closed point $x\in X$ contained in the stratum $Z$.
   Then
  there exists a hypersurface $H\hookrightarrow \P^N_k $ of large degree such that
  $i\colon X\cap H \hookrightarrow X$ is stratified regular away from $x$ and such that
  the schematic intersection $H\cap Z$
  is singular at $x$, i.e.\ $H$ contains the tangent space to $Z$ at $x$.
  \end{itemize}
\end{prop}

\begin{proof}
Part (i) holds by classical theorem of Bertini. Indeed, $i$ is a stratified regular immersion if and only if $H$ intersects each
stratum transversally by Lemma~\ref{lem:stratregtransv}.

Part (ii) follows from  Bertini theorems for hypersurface sections containing a subscheme, see \cite[Theorem~1]{AK79}. \end{proof}

\subsection{Stratified local complete intersection (lci) morphisms}

Let $(X,\mathbf Z_X)$ and $(Y,\mathbf Z_Y)$ be two noetherian schemes with regular
stratifications. We call a morphism which is of finite type $f\colon Y\to X$ {\em stratified lci} if for any
locally given factorization
\begin{equation}\label{sec:stratlci:eq1}
  f=[W\xrightarrow{g} X]\circ [Y\xrightarrow i W]
\end{equation}
with $g$ smooth and $i$ a closed immersion   we have that $i$  is stratified regular once $W$  is endowed with the stratification $g^{-1}(\mathbf Z_X)$.
If $Y$ is not endowed with a stratification a priori but becomes regularly stratified by
$f^{-1}(\mathbf Z_X)$, and, with this stratification, $f$ is stratified lci,  we call $f$
{\it stratified lci.}

\begin{lem}\label{lem:stratlci}
\begin{itemize}
\item[(i)] If for one factorization~\eqref{sec:stratlci:eq1} with $g$ smooth and $i$ a closed
immersion, $i\colon (Y,\mathbf Z_Y) \to (W,g^{-1}(\mathbf Z_X
))$ is a stratified regular immersion,  then $f$ is stratified lci.
\item[(ii)] The composition of stratified lci morphisms is stratified lci.
\item[(iii)] Let $(X,\mathbf Z_X)$ be a regularly stratified scheme and let
$Y\hookrightarrow X$  be a stratified regular immersion.  Then the blow-up $\mathrm{Bl}_Y(X)\to
X$ is stratified lci.
\end{itemize}
\end{lem}

\begin{proof}
We only prove part (iii), since we use it below. The blow-up along a regular immersed center is
lci   and commutes with base change which preserves the normal bundle of the
 regular immersion of the center, see
 ~\cite[Chapitre~I, Th\'eor\`eme~1]{Mic64}.
 So we just have to observe that the blow-up
 $\mathrm{Bl}_{Y\cap \overline Z} (\overline Z)$ is regular for any stratum $Z\in \mathbf
 Z_X$.  Here we use Lemma~\ref{lem:strregpull} and that the blow-up of a regular scheme in a regular
center is regular~\cite[Chapitre~I, Th\'eor\`eme~1]{Mic64}.
\end{proof}

\subsection{Simple normal crossings varieties} \label{subsec:morselem}

 Consider a reduced, separated, equidimensional scheme $X $ of finite type over a perfect field
$k$ of characteristic different from two. Let $X(1),\ldots ,  X(s)$ be the different irreducible components   of $X$, endowed with the reduced
subscheme structure. For   $I=\{i_1,\ldots , i_r\}\subset \{ 1, \ldots , s\}$ a non-empty
subset, we  denote by $$X(I)= X({i_1)}\cap \ldots \cap X({i_r})$$
the schematic intersection.

We call $X$ a {\em simple normal  crossings variety} or {\em strict  normal  crossings variety} ({\em snc variety}) of dimension $n$ if
$X(1),\ldots ,   X(s)$ are smooth  over $k$  of dimension $n$ and for an algebraic closure
$\overline k$ of $k$ and any closed point $x\in X_{\overline k}$ there exists an
isomorphism of $\overline k$-algebras

\begin{equation}\label{sec:sncvar_eq1}
\sO_{X_{\overline k},x}^{\rm h} \cong \overline k[X_0,\ldots ,  X_n]^{\rm h}/(X_0 \cdots X_m)
\end{equation}
 Here $^{\rm h}$ denotes the henselization with
respect to the ideal   $(X_0,\ldots, X_n)$. Note that for an snc variety $X$ the  $X(I)$   are
smooth and equidimensional over $k$ for any  $\varnothing\ne I\subset \{ 1, \ldots , n\}$.

For an snc variety $X$ over $k$ let $\kappa\colon X\to \Z$
be defined by
\[
\kappa(x)= \# \{ i\,|\, x\in X(i) \}.
\]
Then $\kappa$ is upper semi-continuous. We let $\bf Z$ be the set of all connected
components of
$\kappa^{-1}(j) $ for all $j\in \Z$.
 Then $\bf Z$ is a regular stratification of
$X$, which we will always use in the following. For each stratum $Z \in \bf Z$ there exists
a canonical non-empty subset $I_Z\subset \{1, \ldots , n\}$ such that $\overline Z$ is a
connected component of  $X({I_Z})$.
Set $$X^{(a)}= \kappa^{-1}([a+1,\infty)).$$

In the following let  $D$ be smooth of
dimension one  over $k$ and let $\phi\colon X\to D$ be a $k$-morphism, where $X$ is an snc
variety over $k$.

\begin{prop}[Normal crossings Morse lemma]\label{prop:normcrmorsefield}
  Assume $k$ is algebraically closed.   For a stratified Morse function $\phi\colon X\to D$
with a critical point $x\in X$,   there exists a $k$-algebra isomorphism as in~\eqref{sec:sncvar_eq1} and a
$k$-algebra isomorphism $\sO^{\rm h}_{D,\phi(x)}\cong k[T]^{\rm h}$
such that $\phi^\hh_x\colon \sO^{\rm h}_{D,\phi(x)} \to \sO_{X,x}^{\rm h}$ maps $T$
to
\[
X_0 + \ldots + X_m + X_{m+1}^2 + \ldots + X_n^2.
\]

\end{prop}
Here the henselization $^{\rm h}$ of $ k[T]$ is with respect to the ideal $(T)$.

\begin{proof}
  Set $z=\phi(x)$.
  Say $X$  is presented by  \eqref{sec:sncvar_eq1} at $x$. Then after replacing $X$ by an
  \'etale neighborhood of $x$  there exists a smooth map
  \[  h  \colon X\to \Spec k[X_0 , \ldots , X_m]/(X_0 \cdots X_m)=:Y\]
  sending $x$ to the origin $y\in Y$. Together with $\phi$ this defines
  $$g= (h,\phi) \colon X\to Y\times_k D.$$
 We assume $n>m$ and leave the
  similar but easier case $n=m$ to the reader. Then the map $g$ is flat.
  Indeed, by the fiberwise criterion for flatness  \cite[Lemma~00MP]{StPr}
   and   as
  $X$ and
  $Y\times_k D$ are both flat (even smooth)  over $Y$,
we just have to check  that    $g|_{h^{-1}(y)} \colon h^{-1}(y) \to \{y\}\times D$ is flat,
  which holds as this map is  the restriction of the stratified morphism to the stratum $(X_0=\ldots=X_m=0$).

  By assumption the fiber $\phi^{-1}(z)$ is Morse over $k(z)=k$, thus its  restriction to  the stratum $(X_0=\ldots=X_m=0)$ is still Morse. In other words,  the fiber
   $g^{-1}(y,z)$  is Morse over $k=k(y)\otimes_k k(z)$.
    So by Proposition~\ref{prop1:struc_almsm} there
  exists an $\sO_{Y\times D,y\times z}^\hh$-algebra isomorphism
  \[
\sO_{X,x}^\hh\cong k[X_0,\ldots, X_n,T]^\hh/(X_0\cdots X_m, X_{m+1}^2 + \ldots + X_n^2-
\alpha ) (=: A)
\]
with $\alpha $ a non-unit in  $\sO_{Y\times D,y \times z}^\hh\cong k[X_0,\ldots , X_m,T]^\hh/(X_0\ldots
X_m) (=:B)$.    In $A$ the henselization is with respect to   $(X_0,\ldots, X_n,T)$,  in $B$ with respect to  $(X_0,\ldots, X_m,T)$.
Write \[ \alpha= c_0 X_0 + \ldots +c_m X_m + b T \text{ with } c_0, \ldots ,c_m,b \in B. \]

As    $A/(X_0 , \ldots , X_m)$  is the henselian local ring at $x$ of the stratum containing
$x$, this ring is regular.
This means $b\in B^\times$. Reparametrizing   $Y_{i} \leadsto b^{\frac 1 2} Y_{i}$
for $m<i\le n$ and replacing $c_j$ by $b^{-1}c_j$,  we can assume without loss of generality that $b=1$. As
\[ A/(T,X_1,\ldots , X_m)\]   is the henselization at $x$ of $\phi^{-1}(z)\cap
\overline Z$, where $Z$ is a stratum with $x\in \overline Z\setminus Z$, this ring is
regular, since $\phi$ is a stratified Morse function. This means $c_0 \in B^\times$. By a reparametrization we can assume without loss
of generality that $c_0=1$. Similarly, we can assume that $c_1= \ldots =c_m=1$. This finishes the proof.
  \end{proof}

\subsection{Stratified base change for unipotent nearby cycles}\label{sec:stratbciso}

Let the notation be as in~\ref{app_subsec_unip}. Consider  regular schemes $(X,\mathbf
Z_X)$ and $(Y,\mathbf Z_Y)$ of finite type over $\sO$ endowed with regular
stratifications. We assume that the  (connected components of) $X_K$ and $Y_K$ are
strata.

\begin{prop}\label{sec:bc:mainprop}
 For a stratified lci morphism $f\colon Y\to X$ over $\sO$,  the base change map of unipotent nearby cycles
  \[
f^* \uppsi(\Lambda_{X_K} )  \to \uppsi(f^* \Lambda_{X_K}) \in D^\nil(Y,\Lambda)
\]
is an isomorphism.
\end{prop}

\begin{proof}
 By smooth base change we can assume without loss of generality that $f$ is a stratified
 regular immersion.
By Lemma~\ref{app:lem:derivednaka} it is sufficient to show the  isomorphism  in restriction to the  $G$-invariants. This means we have to show
 that
\[
f^* i^* j_* \Lambda_{X_K} \to  i^* j_*  f^*\Lambda_{X_K}
\]
is an isomorphism. Here $i$ is the immersion of the closed fiber and $j$ the immersion of
the generic fiber.  By
the exact triangle~\eqref{eq_app1:fundncf1}   and the fact that the closed fiber is a union of strata,  we only have to show that the
map
\[
  f^* i^! \Lambda_{X} \to i^! f^* \Lambda_{X}
\]
is an isomorphism.
\begin{claim}
  For any closed immersion $\tilde i\colon \tilde X\to X$, where $\tilde X$ is a union of strata, the map
  \eq{eq:proofbc1}{
  f^* \tilde i^! \Lambda_{X} \to \tilde  i^! f^* \Lambda_{X}
}
is an isomorphism.
\end{claim}

Consider an open stratum $Z\subset \tilde X$ and set $Z_1=\overline Z$. Set $Z_2=\tilde X\setminus Z$.
Let $i_1\colon  Z_1\hookrightarrow \tilde X$,  $i_2\colon  Z_2\hookrightarrow \tilde X$,
$i_{12}\colon Z_1\cap Z_2\hookrightarrow \tilde X$ be the closed immersions.
By the exact triangle
\[
i_{12,*}i_{12}^! \tilde i^! \Lambda_X \to i_{1,*}i_{1}^! \tilde i^! \Lambda_X \oplus
i_{2,*}i_{2}^! \tilde i^! \Lambda_X \to    \tilde i^! \Lambda_X \to \cdots ,
\]
 the corresponding triangle for $Y$ and noetherian induction,  we have to show that
\[
 f^* (\tilde i\,  \tilde i_1)^! \Lambda_{X} \to   (\tilde i \, \tilde i_1)^!  f^* \Lambda_{X}
\]
is an isomorphism, which follows from Gabber's absolute purity  theorem
\cite[Theorem~2.1.1]{Fuj02}  as $Z_1 $ and the schematic pullback  $  f^{-1}( Z_1)$ are regular by
Lemma~\ref{lem:strregpull}.
\end{proof}

\section{Semistable  Morse functions}\label{subsec:semistabmor}

   Let $\sO$ be a henselian
discrete valuation ring. We assume  that  $K=\mathrm{frac}(\sO)$ has characteristic zero
and that the residue field $k=\sO/(\pi)$ of $\sO$ is perfect of characteristic different from two.
 A regular  scheme $X$ separated, flat and of finite type
over $\sO$ is called {\em semistable}  if $X_k$ is an snc variety
over $k$. In this section $X$ is assumed to be semistable.  Let $D$ be a
 scheme which is separated, smooth and of finite type   over $\sO$ of relative
dimension one.

Note that our semistable concept  is sometimes called strictly semistable in the
literature.
We always endow a semistable scheme $X$ with the standard stratification with strata being the
connected components of $X_K$ and the strata of $X_k$ as in Section~\ref{subsec:morselem}.

\subsection{Main geometric theorem}

Recall the local structure of semistable schemes.

\begin{lem}
If $k$ is algebraically closed and $x\in X_k$ is a closed point,  there exists an
isomorphism of $\sO$-algebras
\[
\sO_{X,x}^{\rm h}\cong \sO[X_0 , \ldots , X_n]^\hh/(X_0 \cdots X_m-\pi)
\]
for some $0\le m\le n$.
\end{lem}
Here the henselization $^{\rm h}$ is with respect to the ideal $ ( \pi, X_0, \ldots, X_n) $.

Our main theorem about stratified  Morse functions for semistable schemes is the following.

\begin{thm}\label{thm:mainpropermorse}
 Assume that $X$ is proper over $\sO$. Let $\phi\colon X\to D$ be an
$\sO$-morphism such that $\phi_k\colon X_k\to D_k$ is a Morse function. Then
\begin{itemize}
  \item[(i)]
    $\phi_K\colon X_K\to D_K$ is a Morse function;
  \item[(ii)] the specialization map $\mathrm{sp}\colon |X_K|\to |X_k|$ on closed points
    induces a bijection between the critical points $\{x_K\}$ of
    $\phi_K$ and the critical points $\{x_k\}$ of $\phi_k$;
  \item[(iii)] for a critical point $x_K$ of $\phi_K$,  let $m$ be
$\mathrm{codim}_{X_k}(Z)$, where $Z\subset X_k$ is the stratum of the point
$x_k=\mathrm{sp}(x_K)$. Then the schematic closure $S$ of $x_K$ in $X$ is a trait of ramification degree $m+1$ over $\sO$;
\item[(iv)]  $\phi|_S\colon S\to D$ is a closed immersion if  $\mathrm{ch}(k) \nmid m+1$.
\end{itemize}
\end{thm}

Here we use the  word {\it trait} for the spectrum of a henselian discrete valuation ring.

The proof of the theorem relies on the following semistable version of the Morse lemma
 the   proof of which is almost verbatim the same as the proof of Proposition~\ref{prop:normcrmorsefield}.

\begin{prop}[semistable Morse lemma]\label{prop:semistmorselem}
Assume that $k$ is algebraically closed and that $\phi\colon X\to D$ is a $\sO$-morphism
such that $\phi_k\colon X_k \to D_k$ is a Morse function. Let $x\in X_k$ be a closed point.
Then there exist isomorphisms
of $\sO$-algebras
\[
\sO_{D,\phi(x)}^{\rm h}\cong \sO[T]^{\rm h},\quad \sO_{X,x}^{\rm h} \cong \sO[X_0,\ldots ,
X_n]^\hh/(X_0 \cdots X_m- u \pi),
\]
where $u$ is a unit in $ \sO[X_0,\ldots , X_m]^\hh$ such that
\begin{itemize}
\item[(i)] if $x\in X_k$ is non-critical,
then $\phi^{\hh}_x\colon
\sO_{D,\phi(x)}^{\rm h} \to  \sO_{X,x}^{\rm h} $ sends $T$ to
$ X_{m+1}$,
in particular $n>m$;
\item[(ii)] if $x\in X_k$ is critical
then  $\phi^{\hh}_x\colon
\sO_{D,\phi(x)}^{\rm h} \to  \sO_{X,x}^{\rm h} $ sends $T$ to
\[
X_0 + \ldots + X_m + X_{m+1}^2 + \ldots + X_n^2,
\]
where $m=n$ is allowed.
\end{itemize}
\end{prop}
 Here the henselization  $^{\rm h}$ of $\sO[T]$  is with respect to the ideal $ ( \pi, T)$
and that of $\sO[X_0,\ldots , X_n]$ with
respect to the ideal $(\pi, X_0,\ldots, X_n)$.

\subsection{Local version and proof}

Theorem~\ref{thm:mainpropermorse} follows immediately from the following local version.

\begin{prop}\label{prop:mainthmmorse}
  Let $\phi\colon X\to D$ be an
$\sO$-morphism such that $\phi_k\colon X_k\to D_k$ is a Morse function with only one
critical point $x_k \in X_k$. Let $m$ be the codimension in $X_k$ of the stratum
containing $x_k$.
Then after  replacing $X$ by a small Zariski neighborhood  of
$x_k$,  the morphism $\phi_K\colon X_K \to D_K$ has precisely one critical point $x_K\in
X_K$, and it satisfies the following properties:
\begin{itemize}
\item[(i)] $x_K$ is non-degenerate;
\item[(ii)] the schematic closure $S$ of $x_K$ is a trait which is finite of ramification
  index $m+1$ over $\sO$ with $S_k=\{x_k\}$;
  \item[(iii)] the
    morphism $\phi|_S\colon S\to D$ is a closed immersion if $\mathrm{ch}(k) \nmid m+1$.
\end{itemize}
\end{prop}

\begin{proof}
  {\em Case $m=0$}. This is classical and follows from Proposition~\ref{prop1:struc_almsm}.

  \smallskip

  {\em Case $m>0$}.
After base change by $\sO \to \sO^{\rm sh}$ we can assume that $k$ is algebraically
closed.    We can check the uniqueness of $x_K$ and the required properties of the map
  $S\to \Spec \sO$ after formally completing $X$ at $x_k$. Note that the map $S\to \Spec
  \sO$ is finite, because $\sO$ is henselian.     We can also assume without loss of generality that $\sO$ is complete.

Let $J_\phi$ be the Jacobian ideal sheaf of $\phi$ and set $J=J_{\phi, x_k}\widehat{\sO}_{X,x_k}$.  Set $I=(J : \pi^\infty
)$, i.e.\ $I=\{ \alpha \in \widehat{\sO}_{X,x_k} \, |\, \pi^N \alpha \in J \text{ for
  some } N\ge 0\}$. Set
  $$A=\widehat{\sO}_{X,x_k}/ I.$$

 Take a
presentation as in the semistable Morse lemma, Proposition~\ref{prop:semistmorselem}. By
abuse of notation we will identify $I$ and $J$ with their preimage in $\sO\llbracket
X_0,\ldots , X_n \rrbracket$.   We can assume
without loss of generality that $n=m$ by using that the external sum of Morse functions  is a Morse function and using case $m=0$ above.

\begin{lem}\label{main_thmmorse:lem}
There exist elements $u_1,\ldots , u_m\in \sO\llbracket X_0 , \ldots ,X_m \rrbracket $
such that the ideal $I$ is generated by the power series
\begin{align}
\label{lem:maineq1}  &X_0-X_i - u_i X_0 X_i \quad (1\le i\le m),\\
\label{lem:maineq2}  & X_0 \cdots  X_m -u\pi.
\end{align}
Here $u\in  \sO\llbracket X_0 , \ldots ,X_m \rrbracket^\times$ is the unit from the
semistable Morse lemma (Proposition~ \ref{prop:semistmorselem}). In particular $A\ne 0$ is finite flat over $\sO$.
\end{lem}

Before proving Lemma~\ref{main_thmmorse:lem} we observe

\begin{claim}\label{claim:main_morse}
Let $I'\subset \sO\llbracket X_0 , \ldots ,X_m \rrbracket$ be the ideal generated by the
elements~\eqref{lem:maineq1} and~\eqref{lem:maineq2} for arbitrary $u_1,\ldots, u_m \in \sO\llbracket X_0 , \ldots ,X_m \rrbracket $ and for an arbitrary unit $u\in  \sO\llbracket X_0 , \ldots ,X_m \rrbracket$.
Then  $A'= \sO\llbracket X_0 , \ldots
,X_m \rrbracket/I'$ is a discrete valuation ring  which is  totally ramified of degree
$m+1$ over $\sO$.
\end{claim}

\begin{proof}
The  implicit function theorem implies the existence of formal power series
\[ f_1,\ldots, f_n\in \sO\llbracket X\rrbracket \]
with the property $f_i\in X + X^2  \sO\llbracket
X\rrbracket $,  such that $X_i=f_i(X_0) \in A'$. Then the formal power series $$h=X f_1(X)
\cdots f_m(X) -u(X,f_1(X), \ldots , f_m(X)) \pi\in  \sO\llbracket X\rrbracket $$ satisfies
$$h\in X^{m+1} + \pi  \sO\llbracket X\rrbracket  + X^{m+2}  \sO\llbracket X\rrbracket $$
with constant coefficient indivisible by $\pi^2$. So by
the Weierstra\ss\ preparation theorem $h= v \tilde h$  with $v\in  \sO\llbracket
X\rrbracket^\times$ and $\tilde h\in \sO[X]$ an Eisenstein polynomial of degree $m+1$.

As $\tilde h(X_0) = v^{-1}(X_0) h(X_0)=0 \in A'$,  we obtain a surjective $\sO$-algebra
homomorphism
\[
  \theta \colon \sO[X]/(\tilde h) \to A' , \quad X\mapsto X_0,
\]
This implies that $A'$ is finite over $\sO$, but then by~\cite[Theorem~22.6]{Mat86} $A'$ is also flat over
$\sO$.
As $K[X]/(\tilde h)$ is a field by the Eisenstein irreducibility criterion, and  as
$A'\otimes_\sO K\ne 0$, we deduce that $\theta$ is an isomorphism.
\end{proof}

\begin{proof}[Proof of Lemma~\ref{main_thmmorse:lem}]
  Note that $J$ is the  (continuous) Jacobian ideal of the local ring homomorphism
  \[
\sO\llbracket T \rrbracket \to \sO\llbracket T,X_1,\ldots, X_m\rrbracket /(g),
\]
where  $g=X_1\cdots X_m(T-X_1-\ldots - X_m) - u\pi$ is the element from the semistable
Morse lemma in which we made the substitution $X_0 =T-X_1-\ldots -X_m$. So $J$ is the
ideal generated by the $m+1$ elements
\ml{}{
  (1\le i\le m): \
\partial_{X_i} g= -2 X_1 \cdots X_m +  X_1 \cdots\widehat X_i \cdots X_m(T-\sum_{i \ne
    j\ge 1}X_j) - \pi  u u_i
  =\\(X_0-X_i) X_1 \cdots\widehat X_i \cdots X_m- \pi u u_i , \text{ and }
   X_0 \cdots  X_m -u\pi}
in the coordinates $X_0, \ldots , X_m$.
Here the partial derivative $\partial_{X_i}$ is the one killing $T,X_j$ for $ i\ne j\ge
1$ and $u_i= u^{-1}\partial_{X_i} u$.
Multiplying the expression on the right  side  of the equation for $\partial_{X_i}
g$ by $X_0X_i$ and substituting $u \pi$ for $X_0\cdots X_m$  we  obtain the
element~\eqref{lem:maineq1} times $u\pi$.

Let $I'\subset \sO\llbracket X_0 , \ldots ,X_m \rrbracket$ be the ideal generated by the
elements~\eqref{lem:maineq1} and~\eqref{lem:maineq2} with the $u_i$ as above.
We have just   shown $\pi I'\subset J$. A  simple calculation shows that $J\subset I'$.  Therefore, $I=(I':\pi)=I'$, where the
second equality follows from Claim~\ref{claim:main_morse}.  We have shown Lemma~\ref{main_thmmorse:lem}.
\end{proof}

Now part (i) and part (ii) of Proposition~\ref{prop:mainthmmorse} follow by combining
Claim~\ref{claim:main_morse} and Lemma~\ref{main_thmmorse:lem}.

For part (iii) observe that $\phi|_S\colon S\to D$ is given locally by the ring
homomorphism $\sO\llbracket T \rrbracket \to A$ which sends $T$ to
\[
  X_0 +  f_1(X_0)  +
  \ldots + f_m(X_0)  \in (m+1)X_0 +  X_0^2 A
\]
with the notation of the proof of Claim~\ref{claim:main_morse}.
So this ring homomorphism is surjective if $m+1$ is invertible in $\sO$.
\end{proof}

\section{Geometry of Lefschetz pencils}\label{sec:geolefschetz}

\subsection{Stratified Lefschetz pencils}

Let $X$ be a projective scheme over $k$ together with a regular stratification.
Fix a closed immersion $\iota\colon X\hookrightarrow \mathbb
P^N_k$.

A $k$-rational point $V$ of the Grassmannian $\mathrm{Gr}(2,H^0(\mathbb P^N_k,\sO(d))$ is called a {\em pencil
  of hypersurfaces of degree $d$}. The {\em base } $A$ of the pencil is the intersection of
two different hypersurfaces of the pencil. The {\em pencil map} is the canonical map
$\mathbb P^N_k\setminus A \to  \mathbb P (V)\cong \mathbb P^1_k$ sending a point to the
unique hypersurface of the pencil containing it. The {\em compactified pencil map} (or for
short just {\it pencil map}) is the induced morphism $\mathrm{Bl}_A (\mathbb P^N_k) \to \mathbb P^1_k$.

\begin{defn}[Lefschetz pencil]\label{def:stratLefp}
  A pencil of degree $d$ hypersurfaces is called a {\it Lefschetz pencil} for $X$ if
  \begin{itemize}
    \item[(i)] the base $A$
      of the pencil intersects each stratum $Z\hookrightarrow X$ transversally, i.e.\
      $A\cap X\to X$ is a stratified regular immersion of codimension two;
    \item[(ii)] the pencil map $X\setminus( A \cap X) \to \mathbb P^1_k$ is a stratified
      Morse function and has at most one critical
      point per geometric fiber.
    \end{itemize}
\end{defn}

If for a Lefschetz pencil for $X$ we set $\tilde X = \iota^* \mathrm{Bl}_A (\mathbb P^N_k)$ then we obtain
the pencil map $\phi\colon \tilde X\to \mathbb P^1_k$. Note that as $A\hookrightarrow
\P^N_k$ and $A\cap X\hookrightarrow X$ are regular closed immersions of the same codimension ($=2$)
we have an isomorphism $  \mathrm{Bl}_{A\cap
  X}(X)\cong  \iota^* \mathrm{Bl}_A (\mathbb P^N_k)$ \cite[Chapitre I, Th\'eor\`eme~1]{Mic64}.
   The pullback of the
stratification of $X$ to $\tilde X$ is a regular
stratification by Lemma~\ref{lem:stratlci}.
The compactified  pencil map  $\phi \colon \tilde X\to \mathbb P^1_k$ is a Morse function which has no critical
points over $A\cap X$.

\begin{prop}\label{prop:exlefpen}
For $d$ large there is an open dense subset of Lefschetz pencils of $X$ in $\mathrm{Gr}(2,H^0(\mathbb P^N_k,\sO(d))$.
\end{prop}

\begin{proof}
  For simplicity of notation we replace $\iota$ by its composition with the Veronese
  embedding
  \[
    \P^N_k\hookrightarrow \P^{\binom{N+d}{d}-1}_k
  \]
  of degree $d$, so now $\P(V)$ is a line in the
  dual space $\check \P^N_k$. We also assume for simplicity that $k$ is algebraically closed.

    For a smooth closed subscheme $Z\subset X$ let
 $\check Z\subset \check \P^N_k$ be its dual variety,  i.e.\ the closure of the set
 of hyperplanes $H$ with $H\cap Z$ singular.  For $d>1$ the reference \cite[Expos\'e~XVII, Proposition~3.5]{SGA7.2}
 implies
 that $\check Z$ is a  hypersurface and a line $\P(V) $ in $\check \P^N_k$ with associated
 base $A\subset \P^N_k$ is a
 Lefschetz pencil for $\overline Z$ if and only if
 \begin{itemize}
\item[(I)] $A$ intersects $\overline Z$ transversally and
\item[(II)] $\P(V)$ does not intersect $\check Z^{\rm sing}$.
\end{itemize}

 The critical values are $\P(V)\cap \check Z$.   By Proposition~\ref{prop:bertini} we obtain that
 for $d \gg 0$ the dual varieties $\check{\overline Z}$ for strata
 $Z\in \mathbf Z$ are different.  Then any pencil $V$ which satisfies (I) and such that
 $$(\cup_{Z\in\mathbf Z}
 \check Z)^{\rm sing}\cap \P(V)=\varnothing$$
  is   a stratified Lefschetz pencil. The set of those $V$ is clearly
 open and non-empty by the theorem of Bertini.
\end{proof}

\subsection{Semistable  Lefschetz pencils}

Let $X$ be a  projective semistable scheme over $\sO$ of relative dimension $n$. Fix a closed immersion
$\iota\colon X\hookrightarrow \mathbb P^N_\sO$. Recall that we endow $X$ with the standard
stratification as in Section~\ref{subsec:semistabmor}.

An $\sO$-point $V$ of the Grassmannian $\mathrm{Gr}(2,H^0(\mathbb P^N_\sO,\sO(d)) )$ is called a {\it pencil of
hypersurfaces of degree}  $d$. This is the same as a  rank $2$ free $\sO$-submodule of $H^0(\mathbb
P^N_\sO,\sO(d))$ which has a free complementary submodule.

\begin{defn}\label{def:sestLefpen}
If $V_k$ defines a stratified Lefschetz pencil for $X_k$ and  $V_K$ defines a Lefschetz
pencil for $X_K$,  then we say that $V$
defines a {\it semistable Lefschetz pencil for} $X$. The set of {\it critical points} resp.\ {\it critical
values} is the union of the set of critical points resp.\ critical values of  the pencil
maps $\phi_K$ and $\phi_k$.
\end{defn}

Note that if $V_k$ defines a stratified Lefschetz pencil, the base $A$ has the property that
$A\cap X\to X$ is a stratified regular immersion by Lemma~\ref{lem:stratregtransv}. So if
under this assumption we set $\tilde X =\mathrm{Bl}_{A\cap X}(X)= \iota^* \mathrm{Bl}_A (\mathbb P^N_\sO)$ then we get
the  pencil map $\phi\colon \tilde X\to \mathbb P^1_\sO$ which is stratified lci by Lemma~\ref{lem:stratlci}.

Let us recall the properties we showed in Theorem~\ref{thm:mainpropermorse} in the
case  of large residue characteristic.

\begin{thm}\label{thmmain:lefpen}
If ${\rm ch}(k)>n+1$ or $\mathrm{ch}(k)=0$ and $V_k$ defines a Lefschetz pencil for $X_k$ then $V$ defines a Lefschetz
pencil for $X$. Moreover, the set of critical points on $X$ is closed and maps isomorphically (as a
scheme) onto
the set of critical values. Each connected component of the critical points is a trait
which is finite and of ramification index over $\sO$ equal to the number of
irreducible components of $X_k$ it meets.
\end{thm}

In case $0<\mathrm{ch}(k)\le n+1$ we cannot combine Proposition~\ref{prop:exlefpen} and
Theorem~\ref{thm:mainpropermorse} to deduce the existence of a  Lefschetz pencil over
$\sO$, since using them it is not clear whether $\phi_K$ has at most one critical point per
geometric fiber.
However, one easily proves the following proposition.

\begin{prop}
For $d$ large there is an open dense subset in $\mathrm{Gr}(2,H^0(\mathbb P^N_k,\sO(d)))$
such that if $V_k$ lies in this subset,  then $V$ defines a Lefschetz pencils for $X$.
\end{prop}

\begin{proof}
Let $B_F\subset \mathrm{Gr}(2,H^0(\mathbb P^N_F,\sO(d))))$ be   the Zariski closure of all
   non-Lefschetz  pencils over $F$ for $F=K$ or $F=k$. Then
 $B_F$  is not dense in  $\mathrm{Gr}(2,H^0(\mathbb P^N_F,\sO(d)))$ for $F=K$ and $F=k$.
Thus
$\overline{B_K}\otimes_\sO k$
is a proper closed subset of $\mathrm{Gr}(2,H^0(\mathbb P^N_k,\sO(d))))$, so the open subset
\[
\mathrm{Gr}(2,H^0(\mathbb P^N_k,\sO(d))) \setminus (  \overline{B_K}\otimes_\sO k\cup B_k)
\]
has the required properties.
\end{proof}

\section{Reminder on the cohomology of Lefschetz pencils}\label{sec:coholefcl}

In this section we reformulate the classical results about the cohomology of
Lefschetz pencils in  terms of perverse sheaves. This has the advantage that one gets rid
of the unpleasant dichotomy of cases (A) and (B) in~\cite[Expos\'e~XVII]{SGA7.2}.

\subsection{Picard-Lefschetz theory}\label{subsec:pltheocl}

  Let $\ell$ be a prime number
invertible in $k$ and let $\Lambda=\Z/\ell^\nu\Z$ or let $\Lambda$  be an
algebraic field extension of $\Q_\ell$.
Let $f\colon X\to \Spec k$ be a smooth projective  morphism of schemes,   $n=\dim X$. In this section we use the perversity associated to
$\delta_X\colon X\to \Z$, $\delta_X(x)=\dim \overline{\{ x\} }$, see~\ref{app:subsecperv}.

For $\Lambda=\QQl$, a perverse sheaf $\sF\in D^b_c(X,\Lambda)$ is called {\em geometrically semisimple} if
$\sF_{\overline k} \in D^b_c(X_{\overline k},\Lambda)$ is semisimple. For $X$
geometrically connected a geometrically
semisimple perverse sheaf $\sF\in  D^b_c(X,\Lambda)$ has a canonical decomposition
$$\sF= \sF^{\rm c}\oplus \sF^{\rm nc}$$
 into a {\em geometrically constant} part $\sF^{\rm c}= f^* (\sF')[n]$, where
$\sF'\in D^b_c(\Spec k,\Lambda)$ is situated in degree $0$, and a part $\sF^{\rm nc}$
which has no non-trivial geometrically constant subsheaves, see~\cite[Corollaire~4.2.6.2 and following comment]{BBD83}.
 We say that a geometrically semisimple perverse sheaf $\sF$ on $X$
satisfies {\em multiplicity one} if $\sF_{\overline k}$ is a direct sum of pairwise non-isomorphic irreducible
 perverse sheaves.
Let $\sC =\Lambda [n]$ be the  constant
perverse sheaf on $X$.

Fix an immersion $X\hookrightarrow \mathbb P^N_k$. Consider a Lefschetz
pencil with center $A$ as above and $\tilde X= \mathrm{Bl}_{X\cap A}(X)$. Let
$\phi\colon \tilde X\to \P^1_k$ be the pencil map. We call the perverse sheaf
$$\sL =\per
R^0\phi_* \sC  \ {\rm on} \ \P^1_k$$
 the associated
{\em Picard-Lefschetz sheaf}.

\begin{thm}\label{thm:glcllefschetz}
  Assume $\Lambda=\QQl$. Then the following properties are verified.
  \begin{itemize}
\item[(i)]  $\per R^i\phi_* \sC$ is a geometrically semisimple perverse sheaf for all $i\in \Z$;
\item[(ii)]  $\per R^i\phi_* \sC$ is geometrically constant for $i\ne 0$;
\item[(iii)]
$\sL^{\rm nc} = (\per R^0\phi_* \sC)^{\rm nc}$ satisfies multiplicity one.
\end{itemize}
More precisely, if $X$ is geometrically connected then either
\begin{itemize}
\item[(A)]   there
is no skyscraper in $\sL$ and $\sL^{\rm nc}$ is geometrically irreducible
or
\item[(B)]  $\sL^{\rm nc}$ is a sum of one-dimensional
  skyscrapers at the critical values of $\phi$.
  \end{itemize}
\end{thm}

\begin{rmk}\label{rmk:caseAkatz}
By \cite[Th\'eor\`eme~6.3]{SGA7.2} the Lefschetz pencil is of type (A) if the degree of the pencil is
sufficiently large. For $n$ odd type (B) does not occur.
\end{rmk}

From a modern point of view Theorem~\ref{thm:glcllefschetz}(i) is simply the decomposition
theorem for proper morphisms~\cite[Th\'eor\`eme~6.2.5]{BBD83}, while Theorem~\ref{thm:glcllefschetz}(ii) is a simple
consequence of the Weak Lefschetz theorem. The proof of Theorem~\ref{thm:glcllefschetz}(iii) uses  the
local Picard-Lefschetz formula, see~\cite[Expos\'e XV]{SGA7.2}. In the following we recall the part of the local theory
which we use explicitly in this paper.

\medskip

 Consider a point $x\in \P^1_k$. Fix a geometric point $\overline \eta_x$ over the
generic point $\eta_x\in \Spec \sO_{\P^1_k,x}^{\hh}$. Set
$G_x=\mathrm{Gal}(\overline \eta_x /\eta_x)$. Consider
$\overline\jmath_x \colon \overline \eta_x \to \P^1_k$, $i_x\colon \Spec k(x) \to \P^1_k$.
The $\Lambda$-module of pre-vanishing cycles at $x$ is defined by
\[
  V_x := H^{-1}(x, \mathrm{cone}(i^*_x \sL\to  i^*_x  \overline\jmath_{x,*} \overline \jmath^*_x
  \sL )).
\]
Note that all other cohomology groups in degree $\ne -1$  vanish  as $\sL$ is
  perverse. One can deduce this directly from the definition of a  perverse sheaf or
  interpret it as a special case of \cite[Corollaire 4.6]{Ill94} applied to the identity morphism.

  The following proposition is a consequence of the Picard-Lefschetz
  formula~\cite[Expos\'e XV]{SGA7.2}.

\begin{prop}\label{prop:picardlefschetzform}
  Let $x$ be a critical value of $\phi$.
  \begin{itemize}
  \item[$\bullet$]
    If $n$ is even there is a canonical isomorphism $V_x\cong \Lambda (-\frac n 2)$ up to sign.
  \item[$\bullet$]
If $n$ is odd there is a canonical isomorphism $V_x\cong \Lambda(-\frac{n-1}{2})$ up to
sign.
\end{itemize}
\end{prop}
The proposition means that the   homomorphism of $G_x\to
  \mathrm{Aut}(V_x)/\{ \pm \mathrm{id} \}$ is given in terms of a
power of the cyclotomic character. In particular the order of
$\im(G_x)\subset \mathrm{Aut}(V_x)$ divides $2$ if $G_x$ acts trivially on $\Lambda(1)$.

\subsection{The middle primitive cohomology}

The results of this section are due to   \cite[Expos\'e XVIII, Th\'eor\`eme~5.7]{SGA7.2}  for Lefschetz
pencils of type (A). In Proposition~\ref{prop:primmidco} we formulate our result,
and in Remark~\ref{rmk:casebprim} we explain what this means in the novel case of Lefschetz pencils of
type (B).

Let $\Lambda$  be an algebraic field extension of $\Q_\ell$, $\sC=
\Lambda_{X}[n]$, $n=\dim (X)$.
Consider the Lefschetz operator \[ L\colon H^i(X_{\overline k},\sC)\to
  H^{i+2}(X_{\overline k}, \sC (1))\]
associated to the given polarization.
By the Hard
Lefschetz  theorem \cite[Th\'eor\`eme~6.2.13]{Del80} the canonical map
from the kernel of $L$ to its (Tate twisted) cokernel is an isomorphism in degree $0$. We write this
group
\[
  H_{\rm prim}= H^0(X_{\overline k},\sC)^L \xrightarrow{\sim}  H^0(X_{\overline k},\sC)_L  .
  \]
Let $\sL= \per R^0 \phi_* \sC$ be a Picard-Lefschetz sheaf as in Theorem~\ref{thm:glcllefschetz}.

\begin{prop}\label{prop:primmidco}
The $\mathrm{Gal}(\overline k /k)$-module $H_{\rm prim}$ is canonically a direct summand
of $H^0(\P^1_{\overline k}, \sL)$.
\end{prop}

\begin{proof}
Set $X^\circ=X\setminus (X\cap A)$, $\phi^\circ= \phi|_{X^\circ}$. Let $Y\hookrightarrow
X$ be a generic hypersurface section in the fixed Lefschetz pencil, in particular $Y$ is smooth over
$k$. Observe that $Y^\circ=Y\setminus (Y\cap A)$ and $X^\circ\setminus Y^\circ=X\setminus Y$ are
affine. We study  the commutative diagram
\eq{eq:lefprimco}{
\begin{gathered}
  \xymatrix{
H^0(\P^1_{\overline k},\per R^0 \phi^\circ_! \sC) \ar@{^{(}->}[d]_\Xi \ar[r]^-{\stackrel{\Upsilon_c}{\sim}} &  H^0_c(X^\circ_{\overline{k}}, \sC)^L
\ar[r]^-{\stackrel{\Omega_c}{\sim}} & H^0(X_{\overline{k}}, \sC)^L  \ar[dd]^\wr \\
H^0(\P^1_{\overline k},\sL)   \ar@{->>}[d]_\Delta   &   &   \\
H^0(\P^1_{\overline k},\per R^0 \phi^\circ_* \sC) & \ar[l]_-{\stackrel{\Upsilon}{\sim}}
H^0(X^\circ_{\overline{k}}, \sC)_L & \ar[l]_-{\stackrel{\Omega}{\sim}}
H^0(X_{\overline{k}}, \sC)_L
}
\end{gathered}
}
where $\Upsilon$ and $\Upsilon_c$ are induced by the Leray spectral. Below we will show that the horizontal maps are isomorphisms.
We deduce that  $\Xi$ is injective and $\Delta$ is
surjective.
So  we see that the upper part of the diagram induces an injection $H_{\rm prim}\to
H^0(\P^1_{\overline k}, \sL)$ with splitting  $ H^0(\P^1_{\overline k}, \sL) \to H_{\rm prim}$
induced by the lower part of the diagram.

In order to show  that the upper and the lower  horizontal  maps  in~\eqref{eq:lefprimco}  are isomorphisms,  it suffices  by duality to consider the upper part of the diagram.
In the sequel the cohomology groups are meant to be with coefficients in $\sC$,
and we assume that $k=\overline k$.  We now proceed by a series of claims (I)--(VI).

\medskip

(I) {\em $ H^0_c(X^\circ) \to  H^0(X)$ is injective}.

\medskip

This follows from a diagram chase in  the following commutative diagram with exact rows
\[
  \xymatrix{
    H^{-1}(X) \ar[r]  & H^{-1}(A\cap X) \ar[r] & H^0_c(X^\circ) \ar[r] & H^0(X) \\
   H^{-3}(X)  \ar[r] \ar[u]_L  &   H^{-3}(A\cap X) \ar[r] \ar[u]_L^\wr &
   H^{-2}_c(X^\circ)=0 . &
  }
\]
The right vertical Lefschetz arrow is an isomorphism by the Hard Lefschetz theorem. The
vanishing of the right bottom group follows from the exact sequence
\[
 0= H^{-2}_c(X\setminus Y_1)\oplus H^{-2}_c(X\setminus Y_2) \to  H^{-2}_c(X^\circ) \to
  H^{-1}_c(X\setminus (Y_1\cup Y_2)) =0
\]
where $Y_1$ and $Y_2$ are two different hypersurfaces of the fixed Lefschetz pencil.
Indeed, the Weak Lefschetz  theorem implies that the groups on the right  and on the left vanish in
this exact sequence.

\medskip

(II) {\em It holds }
\ga{}{  H^0_c(X^\circ)^L \supset \ker( H^0_c(X^\circ) \to H^0_c(Y^\circ)),   \ \
   H^0(X)^L =\im( H^0_c(X\setminus Y) \to H^0(X)). \notag}

   \medskip

For the first inclusion use that
 $L\colon  H^0_c(X^\circ) \to  H^2_c(X^\circ)(1)$ is the
 composition of
 \[
   H^0_c(X^\circ)\to H^0_c(Y^\circ) \to H^2_c(X^\circ)(1).
 \]
 For the second equality use the exact sequence
 \[
 H^0_c(X\setminus Y)  \to  H^0(X)\to  H^0(Y)
 \]
and the description of $L\colon  H^0(X) \to
 H^0(X)$ as the composition
 \[
   H^0(X)\to H^0(Y) \to H^2(X)(1).
 \]
 in which the second arrow is
injective by  the Weak Lefschetz   theorem as $H^{1}(X\setminus Y)=0$.

\medskip

(III)  {\em The maps  $ \ker ( H^0_c(X^\circ) \to  H^0_c(Y^\circ)) \xrightarrow{\sim}  H^0_c(X^\circ)^L \xrightarrow{\sim} H^0(X)^L  $ are isomorphisms}.

\medskip

The first (injective) map is well-defined by (II).
The second map is injective by (I). From (II) we get a surjection  $  H^0_c(X\setminus Y)
\to H^0(X)^L $, which
factors through these maps as $X\setminus Y= X^\circ\setminus Y^\circ$.

In particular (III)   implies that $\Omega_c$ is an isomorphism.

\medskip

(IV) {\em $0\to H^0(\P^1_k, \per R^0 \phi^\circ_! \sC) \to H^0_c(X^\circ)  \xrightarrow{\beta} H^{-1}
    (\P^1_k , \per R^1 \phi^\circ_! \sC ) $ is exact}.

\medskip

This sequence is induced by the perverse Leray spectral sequence, and the exactness follows
since $H^a(\P^1_k,\sG)=0$ for any perverse sheaf $\sG$ and $a<-1$ and since $\per R^a
\phi^\circ_! \sC =0$ for $a<0$ by the Weak Lefschetz theorem.

\medskip

(V) {\em $ \per R^1 \phi^\circ_! \sC  $ is smooth}.

\medskip

Let $E\subset \tilde X$ be the exceptional divisor of the blow-up $\tilde X\to X$. Then
$E\cong (A\cap X) \times \P^1_k$ and $\phi^E:=\phi|_E\colon E\to \P^1_k$ is given by the
projection, see~\cite[Expos\'e XVIII, 2.]{SGA7.2}.
We consider the exact sequence
\[
\per R^0\phi^E_* \sC|_E   \to \per R^1\phi^\circ_! \sC \to \per R^1\phi_* \sC.
\]
Then the perverse sheaf on the left is smooth by smooth base change and the one on the
right is smooth by Theorem~\ref{thm:glcllefschetz}.

\medskip

(VI) {\em It holds  $\ker (\beta) = \ker ( H^0_c(X^\circ) \to  H^0_c(Y^\circ)) $}.

\medskip

By (V) and proper base change the restriction map
\[
  H^{-1}(\P^1_k,\per R^{1} \phi^\circ_! \sC) \to    H^0_c(Y^\circ)
\]
in injective.

The combination of  (III), (IV)  and (VI) gives the isomorphism $\Upsilon_c$.
\end{proof}

\begin{rmk}\label{rmk:casebprim}
If a geometrically connected $X$ has a Lefschetz pencil of type (B),  Proposition~\ref{prop:primmidco} combined with
Proposition~\ref{prop:picardlefschetzform} yields an isomorphism
$$H_{\rm prim }\cong
\Lambda(-\frac{n}{2}  )^{\oplus r}$$
 where $r$ is the number of critical points, which is canonical up to signs on each
 summand on the right.
 For Lefschetz pencils of type (A) the complement in the direct sum decomposition in
 Proposition~\ref{prop:primmidco} is calculated in~\cite[Expos\'e XVIII,
 Th\'eor\`eme~5.7]{SGA7.2}.
\end{rmk}

\section{Monodromy filtration and weight filtration}\label{subsec:mofil}

\subsection{Monodromy filtration for perverse sheaves}\label{subsec:monfilperv}

Let $k$ be a field. Let $ X,Y$
be separated schemes of finite type over $k$. Let $\Lambda=\Z/\ell^n\Z$ or let
$\Lambda$ be an algebraic field extension of $\Q_\ell$. We use the notation of~\ref{app_subsec_unip}.

Consider a perverse sheaf $\sF \in D^\nil(X,\Lambda)$ with its nilpotent endomorphism $$N\colon \sF
\to \sF(-1)^\Iw.$$ Note that if $\Q_\ell\subset \Lambda$ there is a canonical isomorphism
between the
Iwasawa twist and the Tate twist
$\sF(-1)^\Iw\cong \sF(-1)$ by~\ref{subsec_iwasawa}, so in this case one can also write $N\colon \sF
\to \sF(-1)$.

We use the perverse t-structure from Section~\ref{sec:coholefcl}.
 Consider the two filtrations by perverse subsheaves of $\sF$
 \begin{align*}
   \fil_a \sF&=\ker [ N^{a+1}\colon \sF \to \sF(-a-1)^\Iw ], \\
   \fil^a \sF & = \im[ (N(1)^\Iw)^a \colon \sF(a)^\Iw\to \sF ]
 \end{align*}
 for $a\in \Z$.  We denote by  $\gr_a \sF, \gr^a \sF$ the corresponding graded subquotients.
 Set
 \[
   \gr^b_a \sF=\fil_a\sF\cap \fil^b \sF / ( \fil_{a-1} \sF + \fil^{b+1} \sF)\cap \fil_a\sF\cap \fil^b\sF.
 \]
 The {\em
   monodromy filtration} is the  increasing convolution
 \eq{eq:convmofil}{
   \fil^{\rm M}_a \sF= \sum_{b-c=a} \fil_b \sF \cap \fil^c \sF,
}
 filtration.
See~\cite[Section.~2.1]{Sai03}, where one finds a proof of the following lemma.

\begin{lem}\label{lem:monfiltr}
  \begin{itemize}
    \item[(i)] The canonical map
    \[
\bigoplus_{b-c=a} \gr^c_b \sF \xrightarrow{\sim} \gr^{\rm M}_a \sF
\]
is an isomorphism.
\item[(ii)] $N$ sends $\fil^{\rm M}_a \sF$ to $\fil^{\rm M}_{a-2}\sF(-a)^\Iw$ for all
  $a\in \Z$, and it induces an
  isomorphism
  \[
    \gr^{\rm M}_a \sF \xrightarrow{N^a} \gr^{\rm M}_{-a} \sF(-a)^\Iw
  \]
  for all $a\ge 0$.
  \item[(iii)]
The monodromy filtration is the only finite filtration  of $\sF$ by perverse subsheaves which
satisfies  $\rm (ii)$.
\end{itemize}
\end{lem}

\begin{defn}\label{def:mwproperty}
For $f\colon X\to Y$ a proper $k$-morphism  and for $\sF\in D^\nil(X,\Lambda)$ perverse we say that the  {\em monodromy
  property}  holds (for $\sF$ and $f$) if
\begin{equation}\label{eq:mwproper}
  \fil^{\rm M}_a\, \per R^if_*\sF = \im [  \per R^if_* \fil_a^{\rm M} \sF \to  \per
  R^i f_*\sF   ]
\end{equation}
 in the category of perverse sheaves for all $a,i\in \Z$.
\end{defn}

\begin{rmk}\label{rmk:kashiw}
A variant of a conjecture of Kashiwara~\cite{Kas98} says that for $k$ separably closed and
$\Lambda$ an algebraic extension of $\QQl$,  any   $\sF=\uppsi(\sG)\in  D^\nil(\psi^{-1}(0),\Lambda)\subset  D^\nil(X,\Lambda)$  produced by the unipotent nearby
cycle functor $\uppsi$ for a morphism $X\to \A^1_k$ and a semisimple perverse $\sG\in D^b_c(X,\Lambda)$
satisfies the monodromy property for any proper morphism $f\colon X\to Y$.  If
$\mathrm{ch}(k)=0$, one can deduce this conjecture from the work by T.\
Mochizuki~\cite{Moc07} or by applying the method of  Drinfeld~\cite{Dri01}.  If  $\mathrm{ch}(k)>0$ and $\sG$ is arithmetic in the sense of~\cite[Definition~1.4]{EK20},  one can deduce it from the work of~\cite{Laf02}  and   Gabber~\cite[Section~5]{BB93}.

\end{rmk}

\subsection{mw pure sheaves}

Assume now that $k$ is a finite field, that $\Lambda$ is an algebraic extension of $\Q_\ell$   and that $\sF\in
D^\nil(X,\Lambda)$ is mixed and perverse. Then
there is a canonical weight filtration $\fil^{\rm W}_a \sF \subset \sF$, see \cite[Th\'eor\`eme~5.3.5]{BBD83}.

\begin{defn}\label{def:mwpure}
We call the mixed perverse sheaf $\sF$ {\em monodromy-weight pure} (or {\em
  mw pure}) of weight $w\in \Z$ if $$\gr^{\rm M}_a \sF  \ {\rm is \ pure \ of \ weight} \  w+a \ {\rm for \   all} \
a \in \Z$$ i.e.\ if
$$\fil^{\rm M}_a \sF = \fil^{\rm W}_{a+w} \sF.$$
\end{defn}

Recall the following well-known result, see e.g.  \cite[Theorem~2.49]{FO22} for the case of $X={\rm Spec}(\F_q)$.

\begin{prop}\label{prop:mwpureabelian}
For fixed $w\in \Z$ the mw pure perverse sheaves $\sF\in D^\nil(X,\Lambda)$ of weight $w$
form an abelian  subcategory of all perverse sheaves closed
under extensions.
\end{prop}

\begin{proof}
  Mixed perverse sheaves form an abelian category for which every morphism is strict with
  respect to the weight filtration~\cite[Th\'eor\`eme~5.3.5]{BBD83}. On the other hand for
  a morphism of perverse sheaves $\sF\to \sG$ in $ D^\nil(X,\Lambda)$ which is strict with
  respect to the monodromy filtration, the  monodromy filtration on the
    kernel and cokernel is the induced subspace and quotient filtration. This shows that the mw pure
  perverse sheaves in $ D^\nil(X,\Lambda)$ form an abelian subcategory. Similarly, one
  sees that this category is closed under extensions.
\end{proof}

 Assume in the following that $f\colon X\to Y$ is a proper $k$-morphism.

\begin{prop}\label{prop:degmonspecse}
  If $\sF$ is mw pure
 the monodromy spectral sequence
    \[
E_1^{p ,q} =  \per R^{p+q} f_* \gr^{\rm M}_{-p} \sF  \Rightarrow  \per R^{p+q} f_*  \sF
\]
degenerates at the $E_2$-page.
\end{prop}

\begin{proof}
The perverse sheaf  $\per R^{p+q} f_* \gr^{\rm M}_{-p} \sF$ is pure of weight $w+q$ if  $\sF$ is
mw pure of weight $w$, see~\cite[Corollaire~5.4.2]{BBD83}. So the differentials $d_2,d_3,\ldots$  vanish by
weight reasons.
\end{proof}

\begin{prop}\label{prop:critpwprop}
 If $\sF$ is mw pure of weight $w$ the following are equivalent:
  \begin{itemize}
  \item[(i)] The monodromy property holds for $\sF$ and $f$.
\item[(ii)]  $  \per R^{i}f_*  \sF$ is
  mw pure of weight $w+i$ for all $i$.
\item[(iii)] The map
  \[
N^a\colon E_2^{-a,i+a} \to E_2^{a,i-a}(-a)
\]
is an isomorphism for all $a\ge 0$ and $i\in \Z$.
  \end{itemize}
\end{prop}

\begin{proof}
  (i) $\Rightarrow$ (ii): By  \cite[Corollaire~5.4.2]{BBD83}   the filtration on the right of~\eqref{eq:mwproper} is
  the shifted weight filtration. \\
  (ii) $\Rightarrow$ (iii): As in the proof of Proposition~\ref{prop:degmonspecse} the
  perverse sheaf $ E_2^{-a,i+a}$ is the weight $w+i+a$ graded piece of  $  \per R^{i}f_*
  \sF$, so if the weight filtration agrees up to shift with the monodromy filtration on  $
  \per R^{i}f_*  \sF$ we obtain part (iii) by Lemma~\ref{lem:monfiltr}(ii).\\
  (iii)  $\Rightarrow$ (i): Similarly, this follows from Lemma~\ref{lem:monfiltr}(iii).
\end{proof}

\section{Rapoport-Zink sheaves}\label{sec:rapzi}

The study of the nearby cycle functor for the constant sheaf on a semistable scheme
originated from \cite[Expos\'e~I]{SGA7.1}. In the complex analytic framework a more precise calculation was
envisioned in \cite{Ste76}, which was made precise in the \'etale setting by
\cite[Abschnitt~2]{RZ82}. However the latter approach is not formulated in terms of perverse
sheaves, which was done  later  in \cite[Theorem~3.3]{Sai90} in the setting of $\mathcal D$-modules, and by
\cite{Sai03},  \cite{Car12} in the \'etale setting.
In this section we give a summary of the theory based on the duality of the nearby cycle functor.

We make systematic use of the notion of Iwasawa twist developed by Beilinson, which is
necessary in order to give a
coordinate free description, see Appendix~\ref{sec:appendix}.

\subsection{Constant sheaf}
\label{subsec:rzconst}

Let $k$ be a field. Let $f\colon X\to \Spec k$
be a simple normal crossings variety of dimension $n$. Let $\Lambda$ be $\Z/\ell^\nu\Z$ or
an algebraic extension of $\Q_\ell$.   Set  $\varpi=f^! \Lambda$. We use the notation
introduced in~\ref{subsec:morselem}.
We consider the perverse t-structure on $D^b_c(X,\Lambda )$ as in~\ref{subsec:pltheocl}.

For a  perverse sheaf $\sF\in D^b_c(X,\Lambda)$ we denote by $\fil_{\rm S}^a \sF$ the
largest perverse subsheaf of $\sF$ supported in codimension $\ge a$. Set $\sC =
\Lambda_X[n]$. Note that $\sC$ is perverse as $f$ is a local complete intersection  (lci) \cite[Lemma~6.5]{KW01}.
 In this
subsection we recall the description of its support filtration $\fil_{\rm S} \sC$ from
\cite[Section~1.1]{Sai03}. We refer to Section~\ref{subsec:morselem} for the notation used
below. In particular,   $X^{(a)}$ denotes the union of all codimension $\ge
a$ strata in $X$ and for a set of irreducible components $I$ of $X$ we denote the
intersection of the irreducible components in $I$ by  $X(I)$. These  are closed subschemes of $X$.

\smallskip

 Let $\pi \colon \tilde X\to X$ be the normalization and set
 $$ \Lambda^{(0)}_X= \pi_*
 \Lambda_{\tilde X} \ {\rm and} \ \Lambda^{(a)}_X=\bigwedge^{a+1}_\Lambda
\Lambda^{(0)} \  {\rm for} \  a>0.$$
 Note that if one chooses an ordering $X(1),\ldots ,X(r)$ of the
irreducible components of $X$ then
\[
\Lambda^{(a)}_X \cong \bigoplus_{1\le i_0< \cdots < i_a\le r} \Lambda_{X({i_0}) \cap \cdots \cap
  X({i_a})} ,
\]
so $\Lambda^{(a)}_X[n-a]  $ is perverse.

 The canonical map $\Lambda_X \to  \pi_* \pi^*
\Lambda_X = \Lambda^{(0)}_X$ defines a Koszul complex
\[
 \mathsf{Ko}_X =[  \Lambda^{(0)}_X \to  \Lambda^{(1)}_X \to
    \Lambda^{(2)}_X \to \cdots ],
  \]
  where  $\Lambda^{(0)}_X$ is placed  in degree $0$, which resolves $\Lambda_X$,
  i.e.\ we have a quasi-isomorphism \[\sC \xrightarrow{\sim}\mathsf{Ko}_X[n].\]
  We consider the
 truncation filtration
\[
\fil^a  \mathsf{Ko}_X  = [ 0\to  \Lambda^{(a)}_X  \to  \Lambda^{(a+1)}_X \to
  \cdots ].
\]
By descending induction on $a$ and the exact triangle
\ga{8.1}{
\fil^{a+1}  \mathsf{Ko}_X  \to \fil^{a}  \mathsf{Ko}_X \to \Lambda^{(a)}_X[-a]  \to \fil^{a+1}  \mathsf{Ko}_X [1]
}
we deduce
\begin{lem}\label{lem:constko}
  \begin{itemize}
    \item[(i)]
  $ \fil^{a} \mathsf{Ko}_X[n]$ is perverse for all $a\in \mathbb Z$.
\item[(ii)] The isomorphism $\sC\xrightarrow{\sim} \mathsf{Ko}_X[n]$ induces an isomorphism of perverse subsheaves
  \[ \fil_{\rm S}^a \sC \cong\fil^a \mathsf{Ko}_X [n] \] for all $a\in \Z$.
\item[(iii)]
   There is a canonical isomorphism
\[
 \kappa \colon \Lambda^{(a)}_X[n-a] \xrightarrow{\sim}  \gr^a_{\rm S}\sC
\]
which induces a canonical perfect pairing
\[
\mathsf c^a\colon \gr_{\rm S}^a \sC \otimes \gr_{\rm S}^a \sC(n-a)  \to \varpi.
\]
\end{itemize}
\end{lem}

Following the sign convention of M.~Saito in~\cite[5.4]{Sai88} one should incorporate the sign
$(-1)^{b(b-1)/2}$  in the perfect pairing
\[
 \Lambda_Y[b] \otimes \Lambda_Y[b](b) \to g^! \Lambda
\]
for a smooth scheme $g\colon Y\to \Spec k$ of dimension $b$.
The reason is that we have a shift
$-1$ in the unipotent nearby cycle functor $\uppsi$, so  we need this sign for
Proposition~\ref{prop:RZconstr} below to hold.

\begin{lem} \label{lem:ext}
For any point $x\in X^{(a+1)}$ with $a\ge 0$  it holds
\ga{}{ Hom_{\Lambda}( \Lambda_X^{(a)}[n-a], \sC)_{\overline{ x}}=0.\notag}

\end{lem}
\begin{proof}
We reduce to the case when $\Lambda = \F_\ell$.   We consider the exact sequence of perverse sheaves
\ga{}{0\to \Lambda_X^{(a+1)}[n-a-1]\to {\rm fil}^a_{\rm S} \sC/{\rm fil}^{a+2}_{\rm S}  \sC\xrightarrow{\alpha}  \Lambda_X^{(a)}[n-a]\to 0. \notag}
As $End(\Lambda_X^{(a)}[n-a])=\Lambda_X^{(a)}$,
a    homomorphism  $\sigma \colon  \Lambda_X^{(a)}[n-a]\to {\rm fil}^a\sC/{\rm fil}^{a+2}
\sC$ defined
\'etale locally around  $x$  forces the boundary   map $\Lambda_X^{(a)}[n-a]\to \Lambda_X^{(a+1)}[n-a]$ to vanish on the image of $\alpha \circ \sigma$.
 This  image is a direct sum of $\Lambda_{X(I)}[n-a]$ with $\# I= a+1$. Therefore, on those
 summands, the boundary map has to vanish, but from its description in the Koszul complex
 it does not vanish on any summand, thus there is no such summand and $\sigma$ can only be zero.
On the other hand,
\ga{}{ Hom_{\Lambda}( \Lambda_X^{(a)}[n-a], {\rm fil}^{a+2}_{\rm S} \sC)= 0, \ Hom_{\Lambda}(
  \Lambda_X^{(a)}[n-a], \sC/{\rm fil}^a_{\rm S} \sC)=0,\notag}
as $ \Lambda_X^{(a)}[n-a]$ is a simple perverse sheaf with support of dimension $n-a$ and
there are no such subquotients in the perverse sheaves on the right sides.
Therefore,
\ga{}{ Hom_{\Lambda}( \Lambda_X^{(a)}[n-a] ,  \sC)\cong Hom_{\Lambda}( \Lambda_X^{(a)}[n-a], {\rm fil}^a\sC/{\rm fil}^{a+2} \sC).\notag}
This finishes the proof.
\end{proof}
\begin{prop}\label{prop:endextconst}
 There are canonical isomorphisms of \'etale sheaves on $X$
  \[
    \mathit{Hom}_\Lambda (\fil^a_{\rm S} \sC , \fil^b_{\rm S} \sC ) \cong \begin{cases}
   0 &   \text{ for } b > a \ge 0\\
        \Lambda_{X^{(a)}}  & \text{ for } a\ge b \ge 0
      \end{cases}
  \]
and
  \[
    \mathit{Ext}^1_\Lambda  ( \Lambda_{X(I)}[n-a ] ,  \Lambda_{X(J)}[n-b] ) \cong \begin{cases}
   0 &   \text{ for } |a-b|\ne 1 \\
   \Lambda_{X(J)}  & \text{ for } a+1= b, \ a \ge 0 \\
    \Lambda_{X(I)}(-1)  & \text{ for } a= b+1, \ b\ge 0
 \end{cases}
\]
for $I$ and $J$ sets of irreducible components of $X$ and  $a= \# I -1$, $b=\#J-1$.
\end{prop}

\begin{proof}
We prove the first case in the first isomorphism.
The  complex  $\fil^a_{\rm S} \sC$ is supported  on $X^{(a)}$ in degree $a-n$ by the exactness of the
Koszul complex above and therefore $\mathit{Hom}_\Lambda (\fil^a_{\rm S} \sC , \fil^b_{\rm S} \sC )=0$ for $a<b$.

For the second case in the first isomorphism one has to show that the canonical map
\[
 \Lambda_{X^{(a)}} \to \mathit{Hom}_\Lambda (\fil^a_{\rm S} \sC , \fil^a_{\rm S} \sC )
\]
is an isomorphism. This is clear away from $X^{(a+1)}$ as there  $\fil^a_{\rm S} \sC  =
\Lambda^{(a)}_X$. We prove isomorphy at the stalk of a fixed geometric point $\overline x$
over $x\in X$   by descending induction on $a$ using the exact sequence of sheaves
\[
  \mathit{Hom}_\Lambda (\Lambda^{(a)}_X[n-a] ,\fil^a_{\rm S} \sC ) \to
  \mathit{Hom}_\Lambda ( \fil^a_{\rm S} \sC ,\fil^a_{\rm S} \sC ) \to
  \mathit{Hom}_\Lambda (\fil^{a+1}_{\rm S} \sC ,\fil^a_{\rm S} \sC )
\]
induced by~\eqref{8.1}.
For $x\in X^{(a+1)}$ the stalk at $\overline x$ of the sheaf on the left vanishes by Lemma~\ref{lem:ext} and
the canonical map
\[
  \Lambda\to
\mathit{Hom}_\Lambda (\fil^{a+1}_{\rm S} \sC ,\fil^a_{\rm S} \sC )_{\overline x} =
\mathit{Hom}_\Lambda (\fil^{a+1}_{\rm S} \sC ,\fil^{a+1}_{\rm S} \sC )_{\overline x}
\]
is an isomorphism by the induction assumption, so $\Lambda \to  \mathit{Hom}_\Lambda (
\fil^a_{\rm S} \sC ,\fil^a_{\rm S} \sC )_{\overline x}$ is an isomorphism.

\smallskip

The second isomorphism is immediate if $I=J$, so assume $I\ne J$.
Write $I=K\sqcup I'$ and $J=K\sqcup J'$ and observe
\ga{}{ 1+a-b=1+ \# I'- \# J'  \le 2 \,  \# I' = 2 \, {\rm codim}_{X_J}(X(I)\cap X(J)).  \notag}
If the inequality is strict, then
  $\mathit{Ext}^1_{\Lambda} (\Lambda_{X(I)}[n-a], \Lambda_{X(J)}[n-b])= \mathcal
  H^{1+a-b}_{X(I)\cap X(J)}(\Lambda_{X(J)})=0 $ by purity.
 Else
 $ 1=\# I'+\# J'$.
 If $\# I'=1$ and $\# J'=0$, then by purity
 \[
   \mathit{Ext}^1_{\Lambda} (\Lambda_{X(I)}[n-a], \Lambda_{X(J)}[n-b])\cong \Lambda_{X(I)}(-1).
 \]
  If $\# I'=0$ and $  \# J'=1$, then $\mathit{Ext}^1_{\Lambda} (\Lambda_{X(I)}[n-a], \Lambda_{X(J)}[n-b])=\Lambda_{X(J)}$.
 \end{proof}

\subsection{Rapoport-Zink sheaves}

We use the notation of   \ref{subsec:monfilperv}, \ref{subsec:rzconst} and
of~\ref{app_subsec_unip}.
We consider the corresponding derived category of
$\Lambda$-sheaves with a continuous unipotent $\Z_\ell(1)$-action $D^\nil(X,\Lambda)$, see~\ref{subsec_iwasawa}. Any
object $\sF $ in  $D^\nil(X,\Lambda)$ has a canonical $\Lambda$-linear nilpotent morphism
$N^\Iw\colon \sF \to \sF(-1)^\Iw$. For simplicity of notation we also write $N$ for
$N^\Iw$.

We define the category ${\RZ}^{\rm pre}(X,\Lambda)$ of {\it pre-Rapoport-Zink sheaves}.
Its objects are pairs $(\sRZ , \iota)$, where $\sRZ\in D^\nil(X,\Lambda)$ is a perverse
sheaf and
\[
  \iota\colon
\sC\xrightarrow{\sim} \ker[\sRZ \xrightarrow{N} \sRZ(-1)^\Iw]
\]
is an isomorphism. A morphism $\phi\colon (\sRZ, \iota_\sRZ)\to (\sRZ', \iota_{\sRZ'})$ in ${\RZ}^{\rm pre}(X,\Lambda)$ is a
morphism $\phi\colon \sRZ\to \sRZ'$ in $D^\nil(X,\Lambda)$ such that $\phi\circ \iota_\sRZ =
\iota_{\sRZ'}$. Note that ${\RZ}^{\rm pre}(X,\Lambda)$ is a {\it groupoid} by Lemma~\ref{app:lem:derivednaka}.

The category of {\it Rapoport-Zink sheaves}  ${\RZ}(X,\Lambda)$ is now defined as a full subcategory of  ${\RZ}^{\rm pre}(X,\Lambda)$.

\begin{defn}\label{defn_RZ-sheaf}
We call $(\sRZ,\iota)\in {\RZ}^{\rm pre}(X,\Lambda)$ a {\em Rapoport-Zink sheaf} ({\em
  RZ-sheaf}) if
\begin{itemize}
\item[(i)]{\rm [compatibility of filtrations]} $\iota^{-1}(\fil^a \sRZ)=\fil^a_{\rm S} \sC$
  for all $a\in \Z$;
\item[(ii)]{\rm [polarizability]} there exists
a perfect pairing
  \[
\mathsf p \colon \sRZ \otimes_{\Lambda^\Iw} \sRZ^- \to \varpi(-n)
\]
in $D^\nil(X,\Lambda)$ such that the induced pairing
\[
\gr^a_0 \sRZ \otimes \gr^0_a \sRZ^-  \to \varpi(-n)
\]
coincides with the pairing
\ml{}{
\gr^a_0 \sRZ \otimes \gr^0_a \sRZ^-   \xrightarrow{\mathrm{id} \otimes N^a } \gr^a_0 \sRZ
\otimes \gr^a_0 \sRZ^-(-a)^\Iw \\  \xrightarrow{\stackrel{\iota^{-1}\otimes \iota^{-1}}{\sim}}   \gr^a_{\rm S} \sC
\otimes \gr^a_{\rm S} \sC (-a)   \xrightarrow{\mathsf c^a} \varpi(-n)
\notag}
for all $a>0$,  where $c^a$ in defined in Lemma~\ref{lem:constko} (iii).
\end{itemize}
\end{defn}

Here we use the notation as in~\ref{subsec_iwasawa} for Iwasawa module sheaves and $\varpi=
f^! \Lambda$.

\begin{rmk}\label{rmk:rapzink}
 (i) For a RZ-sheaf $(\sRZ, \iota )$  we have canonical isomorphisms
  \begin{align*}
 & \gr^a_b \sRZ \xrightarrow{\stackrel{N^b}{  \sim}} \gr^{a+b}_0 \sRZ(-b)
   \xleftarrow{\stackrel{\iota}{\sim}} \gr_{\rm S}^{a+b} \sC(-b) ,\\
    &  \gr^{\rm M}_a \sRZ \cong  \bigoplus_{p-q=a} \Lambda^{(p+q)}_X(-p)[n-p-q]  ,  \\
    &  \gr_a \sRZ   \xrightarrow{\stackrel{N^a}{  \sim}} \fil^a\sRZ \cap \fil_0 \sRZ(-1)^a
      \xleftarrow{\stackrel{\iota}{\sim}} \fil^a_{\mathrm{S}} \sC(-a).
  \end{align*}
 (ii) The  local extension class of $ \gr^{a}_b \sRZ $ by
 $\gr^a_{b-1} \sRZ$, resp.  the local extension class of $ \gr^{a}_b \sRZ $ by
 $\gr^{a+1}_{b} \sRZ$  are, using the last identification in (i),  given by   $(1,1,\ldots , 1) $
 on the $\oplus \Lambda_{X(I)}$ factors, resp. on the $\oplus \Lambda_{X(J)}$ factors (forgetting the Tate twists)  where the summands come from
 the formula
  Proposition~\ref{prop:endextconst}.
  The latter is clear by comparing with the extensions
  in $\mathsf C$ and  one deduces the former from
  the polarizability in Definition~\ref{defn_RZ-sheaf}. All other contributions in the
  extension of  $\gr^{\rm M}_a \sRZ$ by $ \gr^{\rm M}_{a-1} \sRZ$ are trivial by Proposition~\ref{prop:endextconst}.
\end{rmk}

\begin{prop}\label{prop:rzpureunique} For $\Lambda$ an algebraic extension of $\Q_\ell$
  the following properties hold.
\begin{itemize}
 \item[(i)]
  For $k$ a finite field a RZ-sheaf in  $\mathrm{RZ}(X,\Lambda)$ is mw pure of weight $n$.
\item[(ii)]  If additionally $X$ is proper over $k$,  then the
  groupoid $\mathrm{RZ}(X,\Lambda)$ is either empty or contractible.
  So in this case our axioms uniquely determine an $\rm{RZ}$-sheaf if it exists.
  \end{itemize}
\end{prop}
 Recall that a groupoid (or to be precise its nerve) is contractible if it
  is non-empty and there is exactly one morphism
  between any two objects.

\begin{proof}
For part (i) observe that    $  \Lambda_{X}^{(a+b)}[n-a-b](-1)^b$
is pure of weight
$n+b-a$. Then the statement follows from
Remark~\ref{rmk:rapzink} and Lemma~\ref{lem:monfiltr}.

\smallskip

For part (ii)   let $(\sRZ,\iota)$ be in $\mathrm{RZ}(X,\Lambda)$. We first show that the
only automorphism of this RZ-sheaf is  the identity by showing that a
$\Lambda^\Iw$-linear $\phi\colon \sRZ \to \sRZ$ which vanishes on the image of $\iota$ has
to be zero. By Remark~\ref{rmk:rapzink} and
Proposition~\ref{prop:endextconst}   we obtain an isomorphism of sheaves
\[
  \mathit{Hom}_{\Lambda} (   \gr_a \sRZ, \gr_b \sRZ ) \cong \begin{cases}
    \Lambda_{X^{(a)}}(a-b) &  \text{ for } a\ge b \ge 0 \\ 0 & \text{ for } 0\le a<b  , \end{cases}
\]
which induces an isomorphism
\[
\mathrm{Hom}_\Lambda (   \gr_a \sRZ, \gr_b \sRZ ) \cong   \begin{cases}
    H^0(X,\Lambda_{X^{(a)}}) &  \text{ for } a= b \ge 0 \\ 0 & \text{ else }   \end{cases}
\]
by weight reasons.
So $\phi$ is strict with respect to the filtration $\fil_a \sRZ$, in other words it
suffices to show that $\gr_a \phi$ vanishes for all $a\in \Z$. One checks this   by
ascending induction using the commutative
diagram
\[
  \xymatrix{
    \gr_a \sRZ  \ar[r]^{\gr_a \phi} \ar[d]_N &  \gr_a \sRZ \ar[d]_N  \\
    \gr_{a-1} \sRZ(-1)  \ar[r]^{\gr_{a-1} \phi} & \gr_{a-1} \sRZ(-1)
  }
\]
in which the vertical maps are injective.

\smallskip

Now we show that up to isomorphism there is at most one RZ-sheaf $(\sRZ,\iota)$ on $X$.
One can build the RZ-sheaf successively along  the monodromy filtration via the extensions
\ga{eq:extrzcl}{
0\to \fil^{\rm M}_{a-1} \sRZ \to \fil^{\rm M}_{a} \sRZ \to \gr^{\rm M}_{a} \sRZ \to 0,
}
where the graded piece on the right is described by Remark~\ref{rmk:rapzink}. We will use
a weight argument to show that this extension class is uniquely characterized by the
description of RZ-sheaves
in Remark~\ref{rmk:rapzink}.

As $\mathit{Hom}_\Lambda ( \gr^{\rm M}_a \sRZ , \gr^{\rm M}_{a-i} \sRZ ) $  has weight
$\le - 2i$ for $i\ge 1$ by Remark~\ref{rmk:rapzink}  and $X$ proper, we obtain
$H^1(X, \mathit{Hom}_\Lambda ( \gr^{\rm M}_a \sRZ , \gr^{\rm M}_{a-i} \sRZ ) )=0$ for
$i\ge 1$  by applying the Hochshild-Serre spectral sequence. So
\[
\mathrm{Ext}^1_\Lambda (  \gr^{\rm M}_a \sRZ,   \gr^{\rm M}_{a-i} \sRZ ) \to H^0(X,
\mathit{Ext}^1_\Lambda (  \gr^{\rm M}_a \sRZ,   \gr^{\rm M}_{a-i} \sRZ )  )
\]
is injective for $i\ge 1$ and the group on the right vanishes for $i>1$ by Proposition~\ref{prop:endextconst} and weight reasons.
This argument demonstrates that the  extension class of $\gr_{a}^{\rm M}\sRZ $ by  $\gr_{a-1}^{\rm M}\sRZ$ is
determined by Remark~\ref{rmk:rapzink} and that the map
\[
\mathrm{Ext}^1_\Lambda (  \gr^{\rm M}_a \sRZ,   \fil^{\rm M}_{a-1} \sRZ ) \to \mathrm{Ext}^1_\Lambda (  \gr^{\rm M}_a \sRZ,   \gr^{\rm M}_{a-1} \sRZ )
\]
is injective. So the extension class~\eqref{eq:extrzcl} is uniquely determined by the properties of
a RZ-sheaf.
\end{proof}

Let $f\colon X\to Y$ be a proper $k$-morphism.  For a RZ-sheaf $(\sRZ,\iota)$ we can write the monodromy spectral sequence
 from Proposition~\ref{prop:degmonspecse} explicitly   as
\eq{eq:rzspecseq}{
  E_1^{p,q} = \bigoplus_{i\ge \max(0,-p)} \per R^{p+q}f_* (\Lambda^{(p+2i)}_X(-i)[n-p-2i] )  \Rightarrow
  \per R^{p+q}f_*(\sRZ)
}
Note that for $\Lambda$ an algebraic extension of $\Q_\ell$ and $k$ finite, $E_1^{p,q}$ is a pure perverse sheaf of
weight $n+q$ by \cite[Corollaire~5.4.2]{BBD83}.

\subsection{Construction via nearby cycles}\label{sec:rzconstnearby}
Let $X$ be a semistable scheme over $\sO$ with generic fiber $j\colon
X_K\to X$ and  special fiber $i\colon X_k\to X$, $n=\dim(X_K)$.
Note that \[\per \mathcal H^{-1}(i^* j_* \Lambda[n+1])=\Lambda[n]\] is the
constant perverse sheaf which we also denote $\sC$, see Remark~\ref{rmk:pervcontsp}.
Set $\sRZ=\uppsi(\Lambda[n+1])$ and let $\iota\colon \sC\to \sRZ^N$ be the isomorphism induced
by the fundamental exact triangle \eqref{eq_app1:fundncf}. Let $\mathsf p\colon \sRZ
\otimes_{\Lambda^\Iw} \sRZ^- \to \varpi(-n)$ be the perfect pairing  defined in~\ref{app:verdierdu}.

The following proposition is essentially shown in~\cite{Sai03}, so we provide only a
sketch of a proof. Another argument in the setting of $\mathcal D$-modules is given in~\cite[Theorem~3.3]{Sai90}.

\begin{prop}\label{prop:RZconstr}
The above $(\sRZ, \iota)$ together with the pairing $\mathsf p$ satisfy the properties of a
Rapoport-Zink sheaf from Definition~\ref{defn_RZ-sheaf}.
\end{prop}

\begin{proof}[Proof sketch]
It is shown in  \cite[Expos\'e~I]{SGA7.1}  that $N$ acts trivially on $\sH^a(\sRZ)$ for all $a\in \Z$. In
\cite[Lemma~2.5]{Sai03}
 it is shown that  the fundamental exact
triangle~\eqref{eq_app1:fundncf} induces the horizontal quasi-isomorphisms in the
commutative diagram for $a\ge -n$
\[
  \xymatrix{
    \sH^{a+1}(\sRZ)[-a-1] \ar[r]^-\sim \ar[d]^N & [  0 \ar[r] \ar[d] &  \sH^{a+1}(i^*j_* \sC(1
 ))   \ar[r]^-{\{ \pi \} } \ar[d]_{\rm id} &  \cdots]    \\
 \sH^a(\sRZ)[-a](-1) \ar[r]^-\sim & [  \sH^a(i^*j_* \sC) \ar[r]^-{\{ \pi \} } &  \sH^{a+1}(i^*j_* \sC(1
 ))   \ar[r]^-{\{ \pi \} } &  \cdots]
  }
\]
Here $\{ \pi \}\in H^1(\Spec K ,\Lambda(1))$ is the tame symbol of $\pi$. Note that by the
purity isomorphism
\eq{eq:purisoproof}{ \sH^{a-n}(i^* j_* \sC)(1) \cong
  \Lambda^{(a)}_X(-a) }
the objects in the diagram are perverse sheaves and the right
vertical map is a monomorphism. So by definition of the filtration $\fil_a$ we obtain that $\fil_a \sRZ= \tau^{\le a-n} \sRZ$.

The
 isomorphisms
\[
  \gr^0_a \sRZ\xrightarrow\sim \sH^{a-n}(i^*j_* \sC)(1)[n-a]
  \stackrel{\eqref{eq:purisoproof}}{\cong}  \Lambda^{(a)}_X(-a)[n-a]
\]
are dual to the
isomorphisms
\[
 \Lambda^{(a)}_X(n)[n-a]  \stackrel{\kappa}{\cong}  \gr_{\mathrm S}^a \sC(n) \xrightarrow{\stackrel{\iota}{\sim}} \gr^a_0 \sRZ(n)
\]
 in view of~\eqref{app:eq1comdual} and the definition of the purity isomorphism.

 Finally, the purity isomorphism~\eqref{eq:purisoproof} also induces the right vertical
isomorphism of complexes in the commutative diagram in $D^b_c(X_k,\Lambda)$
\[
  \xymatrix{
  \sH^{-n}(\sRZ)[n] \ar@{=}[d] &   \ar[l]^-{\iota}_-\sim [ & \mathsf{Ko}_X[n] \ar[d]^\wr && ] \\
    \sH^{-n}(\sRZ)[n] \ar[r]^-\sim & [  0 \ar[r] & \sH^{-n}(i^*j_* \sC(1)) \ar[r]^-{\{ \pi \} } &  \sH^{1-n}(i^*j_* \sC(2
 ))   \ar[r]^-{\{ \pi \} } &  \cdots ]
  }
\]
This proves part (i) and (ii) of Definition~\ref{defn_RZ-sheaf} for $(\sRZ,\iota)$ and the
polarization $\mathsf p$.
\end{proof}

\begin{prop}\label{prop:grorzuni}
The nearby cycles $R\Uppsi_{X/\sO} (\Lambda)$ are unipotent and in particular tame, i.e.\
the canonical map $\uppsi( \Lambda) \to R\Uppsi_{X/\sO} (\Lambda)[-1]$ is an isomorphism.
\end{prop}

 Tameness is shown in \cite[Satz~2.23]{RZ82}, for a simple proof see \cite[Theorem~1.2]{Ill04}. Under the assumption of
tameness and a semistable case of absolute purity~\cite[Korollar 3.7]{RZ82},  the unipotence was shown much earlier by Grothendieck, see \cite[Expos\'e~I]{SGA7.1}.

\section{Cohomology of semistable Lefschetz pencils}\label{sec:cohosemistlef}

 In this  section we establish a close connection between the monodromy-weight conjecture
 and properties of the cohomology of semistable Lefschetz pencils. We assume that
 $k$ is a finite field.

\subsection{The monodromy-weight conjecture}

For $X$ a projective semistable scheme over $\sO$ we know  by
Proposition~\ref{prop:grorzuni} that the inertia subgroup of $\mathrm{Gal}(\overline K
/K)$ acts unipotently on
\[
  H^i(X_{\overline K} , \Q_\ell) \cong H^i(X_{\overline k}, R\Uppsi_{X/\sO} \Q_\ell )
\]
for all $i\in \Z$, which means that it  is given
in terms of a nilpotent operator
\[
  N\colon H^i(X_{\overline K} , \Q_\ell)\to   H^i(X_{\overline K} , \Q_\ell(-1)),
\]
see~\cite[Expos\'e~I]{SGA7.1}  and~\ref{subsec_iwasawa}.

The operator $N$ induces the monodromy filtration $\fil^{\rm M}
H^i(X_{\overline K} , \Q_\ell)$ as in~\eqref{eq:convmofil}.
We say that $H^i(X_{\overline K} , \Q_\ell)$ is {\em mw pure of weight $i$} if $\gr^{\rm
  M}_a H^i(X_{\overline K} , \Q_\ell)$ is pure of weight $a+i$ as a
$\mathrm{Gal}(\overline k/k)$-module.  It is equivalent to request that its associated
perverse sheaf in $D^\nil(\Spec k,\Q_\ell)$ is mw pure of weight $i$ in the sense of Definition~\ref{def:mwpure}.

Consider the following property depending on an integer $n\ge 0$.

\medskip

 $ ({\bf mw} )_n\;$ For all projective, semistable $X$ over $\sO$ with  $\dim(X_K)\le n$ the
cohomology groups  $H^i(X_{\overline K} , \Q_\ell)$ are mw pure of weight $i$ for all
$i\in \Z$.

\medskip

\begin{rmk}
  \begin{itemize}
   \item
  Deligne conjectures that $ ({\bf mw} )_n$ holds for all $n$  \cite[Section~8.5]{Del70}   (so called
 {\em  monodromy-weight conjecture} in the semistable case). In fact the general
 monodromy-weight conjecture follows from the semistable case by de Jong's alteration theorem. For a general exposition, see ~\cite[Section~3]{Ill94}.
  \item
Grothendieck's degeneration theory of abelian varieties essentially implies $ ({\bf mw} )_1$, a
simplified argument was given by Deligne~\cite[Expos\'e I, Section 6, Appendice]{SGA7.1}.
\item
Using their spectral
sequence~\eqref{eq:rzspecseq} Rapoport--Zink showed  $ ({\bf mw} )_2$
\cite[Satz~2.13]{RZ82}.
\item
In  equal
positive characteristic the conjecture is a theorem which is essentially due to  Deligne,
see  \cite[Th\'eor\`eme~(1.8.4)]{Del80} and~\cite{Ito05}.

\item
  The monodromy-weight conjecture was shown  by Scholze  for set theoretic complete
  intersections in $\P^N_K$ or more generally in a projective toric variety in~\cite[Theorem~9.6]{Sch12}.
See \cite[after Conjecture~1.13]{Sch12} for  a summary of further results for special varieties.
\end{itemize}
\end{rmk}

\subsection{A Lefschetz pencil approach} \label{sec:LP}

Set $\Lambda=\QQl $.
For schemes of finite type over $\sO$ we use the  perversity as in~\ref{app:subsecperv}.
In particular for $X$ regular and flat over $\sO$ with $\dim(X_K)=n$ the sheaf
$\sC =\Lambda_{X}[n+1]$ is perverse. Recall that the unipotent
nearby cycle functor \[\uppsi\colon
D^b_c(X_K,\Lambda) \to D^\nil(X_k, \Lambda)\] maps  perverse sheaves to perverse sheaves, \ref{app_subsec_unip}.

Let $\phi \colon \tilde X\to \P^1_\sO$ be a semistable Lefschetz pencil as in
Definition~\ref{def:sestLefpen}. In this section we  study the cohomological
degeneration over $\sO$ of the perverse Picard-Lefschetz sheaf
$\sL_K= \per R^0 \phi_{K, *}\sC \in D^b_c(\P^1_K,\Lambda)$. In Section~\ref{sec:tamepl} we
study the tameness of $\sL_K$. Let  $\sL_k:= \uppsi(\sL_K)$ be the unipotent nearby cycles
of $\sL_K$, see~\ref{app_subsec_unip}.

\begin{lem}\label{lem:unipopl}
The perverse sheaf $\sL_K$ is unipotent along $\P^1_k$, i.e.\  \[\uppsi(\sL_K)
\xrightarrow \sim   R\Uppsi_{\P^1_\sO / \sO}(\sL_K)[-1]  \] is an isomorphism.
\end{lem}

\begin{proof}
Consider the RZ-sheaf $(\sRZ,\iota)$  on $\tilde X_k$ as
in~\ref{sec:rzconstnearby} with \[\sRZ= \uppsi (\sC)\xrightarrow \sim R\Uppsi_{\tilde
    X/\sO}(\sC)[-1],\]
where the
last isomorphism is Proposition~\ref{prop:grorzuni}.
So  by proper base change
\eq{eq:h0pl}{
 R\Uppsi_{\P^1_\sO / \sO}(\sL_K)[-1]   \cong
   R\phi_{k, *} \sRZ
}
has a unipotent inertia action.
\end{proof}

In \cite{Del74}  Deligne uses Lefschetz theory to prove the Weil conjectures over a finite field, so it is
natural to try to use the following proposition to  proceed analogously over the local field $K$.

\begin{prop}\label{prop:redmwlef}
Assume that $n$ is fixed and that $ ({\bf mw} )_{n-1}$ holds. Let $X$ be a projective,
semistable scheme over $\sO$ with $\dim (X_K)=n$. Then $H^a(X_{\overline K},\Lambda)$ is mw pure of weight
$a$ for all $a\in \Z$ if there exists a semistable Lefschetz pencil $\phi\colon \tilde X\to
\P^1_\sO$ such that
\eq{eq:h0plll}{
H^{-1}(\P^1_{\overline K}, \sL_K) \cong H^0(\P^1_{\overline k},   \sL_k)
}
is mw pure of weight $n$ for the Picard-Lefschetz sheaves $\sL_K= \per R^0 \phi_{K, *}
\sC$ on $\P^1_K$ and $\sL_k = \uppsi(  \sL_K)$ on $\P^1_k$.
\end{prop}

\begin{proof} Recall that  a direct summand of  a  mw pure
  structure is again mw pure, Proposition~\ref{prop:mwpureabelian}.
  The isomorphism~\eqref{eq:h0plll} follows from proper base change and
Lemma~\ref{lem:unipopl}.
   Fix a semistable Lefschetz pencil as in the proposition with associated Lefschetz
   operator \[L\colon H^i(X_{\overline K},\sC ) \to H^{i+2}(X_{\overline K},\sC(1) ).\]
   Let
   $Y_K\hookrightarrow X_K$ be a smooth fiber of the pencil over $K$. By the Hard and Weak Lefschetz theorems the restriction map
  \[
H^i(X_{\overline K},\sC)\to H^i(Y_{\overline K},\sC)
\]
is an isomorphism for $i<-2$ and a split injection for $i=-2$ with inverse the composition of
\[
H^{-2}(Y_{\overline K},\sC) \to H^0(X_{\overline K},\sC(1)) \xleftarrow{\stackrel{L}{\sim}}   H^{-2}(X_{\overline K},\sC).
\]
So we conclude by the Lefschetz decomposition and by Proposition~\ref{prop:primmidco}
($\sC$ has a different shift there).
\end{proof}

\subsection{Main cohomological theorem} \label{ss:mct}

By Proposition~\ref{prop:rzpureunique} the Rapoport-Zink sheaf $\sRZ$ as in the proof of Lemma~\ref{lem:unipopl} is mw pure of weight $n$, so in order to check
whether~\eqref{eq:h0plll} is mw pure of weight $n$ it would suffice to answer positively  the
following two questions  in view of Proposition~\ref{prop:critpwprop}. Recall that the monodromy property is defined in Definition~\ref{def:mwproperty}.

\begin{itemize}
\item[(I)] Is $\sL_k$ mw pure of weight $n$?
\item[(II)]  Does the monodromy property hold for
$\sL_{\overline k}$ and $\P^1_{\overline k}\to  \Spec \overline k $?
\end{itemize}

In Theorem~\ref{thm:mainmwind} we give a positive answer to (I) assuming the monodromy-weight
conjecture is known in smaller dimensions. Question (II) can be seen as an arithmetic
variant of a conjecture of Kashiwara~\cite{Kas98}, see~\ref{subsec:arithkashi}, about the topology of complex varieties.
The complex version  has been proved by T.\ Mochizuki~\cite{Moc07}. We defer the  study of the arithmetic variant  to a forthcoming work.

\begin{thm}\label{thm:mainmwind}
  Assume  $ ({\bf mw} )_{n-1}$ holds. Then the following properties are satisfied.
  \begin{itemize}
\item[(i)]  The sheaf $\sRZ$ satisfies the monodromy property for $\phi_k$ in the sense of
  Definition~\ref{def:mwproperty}.
 In particular $\sL_k$ as in  Proposition~\ref{prop:redmwlef} is mw pure of weight $n$ and the
 monodromy graded pieces $\gr^{\rm M}_a \sL_{\overline k}$ are  semisimple for all
$a\in  \Z$.
\item[(ii)] The non-constant part of $\gr^{\rm M}_a \sL_{\overline k}$ satisfies multiplicity one in the sense
of~\ref{subsec:pltheocl} for all $a\in \Z$.
\end{itemize}
\end{thm}

\begin{cor}\label{cor:mainredmw}
 Assume  $ ({\bf mw} )_{n-1}$ holds. In order to show   $ ({\bf mw} )_{n}$ it would
 ``suffice''  to show the monodromy property for $\sL_{\overline k}$ and the morphism $\P^1_{\overline
   k}\to \Spec \overline k$.
\end{cor}

\begin{proof}[Proof of Corollary~\ref{cor:mainredmw}]
By Theorem~\ref{thm:mainmwind} the perverse sheaf $\sL_k$ is mw pure of weight $n$. Then
by Proposition~\ref{prop:critpwprop} the monodromy property for  $\sL_{\overline k}$ and for  the morphism $\P^1_{\overline
   k}\to \Spec \overline k$ implies that $H^0(\P^1_{\overline k}, \sL_{\overline k})$ is
 mw pure of weight $n$. We finish the proof using Proposition~\ref{prop:redmwlef}.
\end{proof}

\begin{proof}[Proof of Theorem~\ref{thm:mainmwind}]
We check the mw property
  for $\sRZ$ and $\phi_k$ via the criterion of Proposition~\ref{prop:critpwprop}(iii) for
  $f=\phi_k$ and $\sF=\sRZ$.
  By the characterizing property of the monodromy filtration, Lemma~\ref{lem:monfiltr}(ii),
  \eq{eq:NE_1}{
N^a\colon E_1^{-a,i+a} \xrightarrow{\sim} E_1^{a,i-a}(-a)
  }
is an isomorphism  for $a>0$ in view of the definition of the spectral sequence. By Remark~\ref{rmk:rapzink} and by
Theorem~\ref{thm:glcllefschetz} $E_1^{p,q} $ is geometrically constant for $p+q\ne 0$. So
the non-constant part of the $d_1$-differential vanishes, which means
\eq{eq:mainthmncc}{
E_2^{p,q} = (E_1^{p,q})^{\rm nc} \oplus H\left(   ( E_1^{p-1,q})^{\rm c}\xrightarrow{d_1}  (
E_1^{p,q})^{\rm c}  \xrightarrow{d_1}  ( E_1^{p+1,q})^{\rm c} \right).
}
Recall that the upper indices `c' and `nc' stand for the geometrically constant part and
the geometrically  non-constant part as in~\ref{subsec:pltheocl}.

On the first summand on the right  of~\eqref{eq:mainthmncc} the map $N^a$ is an isomorphism in the bidegrees as in~\eqref{eq:NE_1} as
observed above. The second summand is geometrically constant, so we only need to check
that the map of perverse sheaves in $D^\nil(x,\Lambda)$
\eq{eq:NE_2}{
N^a\colon  E_2^{-a,i+a} |_x [-1] \to E_2^{a,i-a}(-a)^\Iw|_x [-1]
}
is an isomorphism for $a>0$, where $x\in\P^1_k$ is a closed non-critical value, which we
will assume to be $k$-rational after possibly
replacing $k$ by a finite extension.

Lift $x$ arbitrarily to an $\sO$-point $s\colon \Spec\sO \to \P^1_\sO$.
By Lemma~\ref{lem:stratregtransv}
\[Y:=\phi^{-1}(s)\hookrightarrow \tilde X\] is a stratified regular
immersion of semistable schemes, see~\ref{subsec:stratregim}, where $\tilde X$ and $Y$ have the standard
stratification as in Section~\ref{subsec:semistabmor}. Then by Proposition~\ref{sec:stratbciso} we get the
base change isomorphism of perverse sheaves
\[
  \sRZ|_{Y_k}[-1]\cong\uppsi
  (\sC|_{Y_K}[-1])
\]
Note that
\[
(\fil^{\rm M}_a \sRZ)|_{Y_k}[-1] = \fil^{\rm M}_a (\sRZ|_{Y_k}[-1]),
\]
because no irreducible perverse constituent of $\sRZ$ is supported on $Y_k$.
In particular the spectral sequence in~\eqref{eq:NE_2} is nothing but the
monodromy spectral sequence for $H^*(Y_{\overline K},\Lambda)$ with respect to the model $Y$.
Finally,
by our assumption  $ ({\bf mw} )_{n-1}$  we get the isomorphy of~\eqref{eq:NE_2}. This proves part (i).

\smallskip

Part (ii) is then clear as the critical values of $\phi_k$ coming from different  strata $Z$ of $\tilde
X_k$ are disjoint, but these critical values are just the non-smooth loci of  $\per R^0 (\phi_k|_{\overline Z})_* \Lambda
[\dim Z]$. As the non-constant parts of these perverse sheaves satisfy multiplicity one by
Theorem~\ref{thm:glcllefschetz} individually, so does their direct sum.
\end{proof}

\begin{ex}
If case (B) in Theorem~\ref{thm:glcllefschetz} holds for the Lefschetz pencil
$\phi_{\overline K}\colon X_{\overline K}\to \P^1_{\overline K}$,  the
monodromy property for   $\sL_{\overline k}$ is clear, because this perverse sheaf  is then a direct sum of a constant
perverse sheaf and of  skyscraper sheaves.
\end{ex}

\subsection{Arithmetic Kashiwara conjecture}\label{subsec:arithkashi}

Motivated by the  Kashiwara conjecture in complex geometry~\cite{Kas98} we suggest that
the following arithmetic Kashiwara conjecture  holds.

Let $\sO$ be a strictly henselian discrete valuation ring. We do not have to assume that
$K=\mathrm{frac}(\sO)$ has characteristic zero or that the residue field $k$ is perfect in
this subsection. Let $X$ be a proper scheme over $\sO$.

\begin{conj}\label{conj:arkash}
For a perverse sheaf $\sF_K\in D^b_c(X_K,\QQl)$ such that its base change $\sF_{\overline
  K}$ to a separable
closure $\overline K$ of $K$ is semisimple and arithmetic (in the sense
of~\cite[Definition~1.4]{EK20}) the following holds for the perverse sheaf $\sF_k=\uppsi (\sF_K)$
\begin{itemize}
\item[(i)] $\gr^{\rm M}_a \sF_k$ is semisimple for all $a\in \Z$,
\item[(ii)] $\sF_k$ satisfies the monodromy property  with respect to the morphism $X_k\to
  \Spec k$, see~Definition~\ref{def:mwproperty}.
\end{itemize}
\end{conj}

The condition of being arithmetic in the conjecture cannot be omitted by a counterexample
of T.\ Mochizuki. In equal characteristic one can establish many cases of the conjecture
by a spreading argument and the techniques mentioned in Remark~\ref{rmk:kashiw}. We defer
the study to a   forthcoming work, a special case of this spreading is  presented in~\cite{Zoc24}.

As a corollary to the results in~\ref{ss:mct} we obtain

\begin{cor}
The arithmetic Kashiwara conjecture for $X=\P^1_{\sO}$ and for mixed characteristic $\sO$ would imply the monodromy-weight conjecture.
\end{cor}

  \begin{rmk}
There exists a version of the monodromy-weight conjecture for the operator $N$ on the
monodromy spectral sequence over any strictly henselian discrete valuation ring $\sO$,
see~\cite[Conjecture 1.2]{Ito05}. By~\cite[Remark 6.2]{Ito05} one can reduce this conjecture
to the  case in which the residue field $k$ of $\sO$ is a perfection of a finitely
generated field. An immediate generalization of our argument  shows that   the arithmetic Kashiwara
conjecture would also imply this version of the monodromy-weight conjecture.
  \end{rmk}

\section{Tameness of the cohomology  of semistable  Lefschetz pencils}\label{sec:tamepl}

\subsection{Reminder on  log structures}

For a background on  log geometry see~\cite{Ogu18}.
In this subsection we do not assume that the residue field of $\sO$ is perfect.
  Let $X$ be a scheme flat and of finite type over
 $\sO$.
Let $j\colon X_K\to X$ be the open immersion of the generic fiber.
Let $M\to \sO_X$ be a log structure.
We always endow $\Spec \sO$ with the {\it canonical log structure} defined by  $\N\to \sO$, $1\mapsto \pi$.
If the fine and saturated (fs)  log scheme $(X,M)$ is log smooth over $\sO$,  then it is log regular, in particular $X$ is normal.
If additionally the trivial locus of the log structure is $X_K$ then $M= \sO_X \cap j_*
\sO_{X_K}^\times$, so the log structure $M$ is not an extra datum. If no explicit log
structure is given on  $X$ we say that it is  {\it log smooth} over
$\sO$ if it is fs and log smooth over $\sO$  with respect to
the log structure  $M= \sO_X \cap j_*
\sO_{X_K}^\times$.
By abuse of notation we also call a scheme log smooth if it is pro \'etale
over a log smooth scheme over $\sO$.

Our basic source of log smooth schemes over $\sO$ stems from the following   lemma, which
is a consequence of~\cite[Theorem~III.2.5.5,  Example~IV.3.1.17]{Ogu18}.

\begin{lem}\label{lem:logslogs}
The morphism of fs log schemes $\Spec \Z[\N^m]  \to \Spec \Z [\N ]$ induced by the monoid homomorphism $1\mapsto (1,\ldots ,
1)$ is log smooth and saturated.
\end{lem}

Recall that for  a saturated morphism of fs log schemes the  base change in the category of
log schemes and in the category of fs log schemes coincide~\cite[Proposition 2.5.3]{Ogu18}.

\begin{ex}\label{tame:exlogsm1}
  Let $X=\Spec \sO [X_1,\ldots ,X_n]/(X_1 \cdots X_m -\pi^e)$ with $1\le m\le n$ and
  $e>0$. Then $X$ is log smooth over $\sO$. This follows from Lemma~\ref{lem:logslogs}
  since $X$ is smooth and strict over the log smooth  schemes
  $ \Spec (\Z[\N^m ] \otimes_{\Z[\N]} \sO )$, where we use from
  left to right the monoid homomorphisms $\mathbb N \to \N^m $, $1\mapsto (1,\ldots , 1)$,   $\N\to
  \sO$, $1\mapsto \pi^e$.
\end{ex}

\begin{ex}\label{tame:exlogsm2}
Let \[X=\Spec \sO [X_1,\ldots ,X_n,T_1,T_2]/(T_1T_2-\pi,  X_1 \cdots X_m - T_1^dT_2^e)\] with $1\le m\le n$ and
$e,d>0$.
Then $X$ is log smooth over $\sO$. This follows from Lemma~\ref{lem:logslogs}
since $X$ is smooth and strict over the log smooth scheme
\[
\Spec (  \Z[\N^m] \otimes_{\Z[\N]}    ( \Z[\N^2] \otimes_{\Z[\N]} \sO ) ).
\]
where we use from left to right the monoid homomorphisms  $\mathbb N \to \N^m $, $1\mapsto (1,\ldots , 1)$, $\N\to \mathbb N^2$, $1\mapsto
(d,e) $,  $\N\to \mathbb N^2$, $1\mapsto (1,1) $ and $\mathbb
N\to \sO$, $1\mapsto \pi$.
\end{ex}

\subsection{Reminder on  tame coverings}

In this subsection we do not assume that the residue field of $\sO$ is perfect.
Tame \'etale coverings have been studied systematically for the first time in~\cite[Expos\'e XIII]{SGA1} in the
form of the so called {\em Grothendieck-Murre tameness}.
The notion of {\em curve tameness} was first
studied by \cite{Wie08} and later in \cite{KS10}.

Consider an excellent  discrete valuation ring $A$, $K=\mathrm{frac}(A)$, and a
finite field extension $K\subset L$. Let $B\subset L$ be the integral closure of $A$. Then
$L$ is a {\em tame extension} of the discretely valued field $K$ if the residue field
extensions of $A\subset B$
are separable and the ramification indices at all maximal ideals of $B$ are coprime to the residue characteristic.

Let now $X$ be a regular, separated scheme of finite type over  $K=\mathrm{frac}(\sO)$.
Consider a proper normal scheme $\overline X$ over $\sO$ which contains $X$ as an open
subscheme.
 Consider an \'etale covering $f\colon X'\to X$. For a
point $x\in \overline X$ of codimension one  we call $X'\to X$ tame over $\overline{\{ x\}}$ if
the normalization $\overline X'$ of $\overline X$ in $X'$ is tame in the above sense with respect to the
discrete valuation defined by $x$.

We say that $ X'\to \overline X$ is {\em curve tame} (or for short just {\em tame}) if for any closed point
$x\in  X$ the finite morphism $( X'\times_X x)\to x$ is  tame  with respect to the unique extension of the discrete valuation from
$K$ to $k(x)$. The following lemmas are a consequence of~\cite[Theorem.~4.4]{KS10}.

\begin{lem}\label{lem:tamepullb}
  If $\sO\subset \sO'$ is a finite extension of henselian discrete valuation rings and
  $U\subset X$ is a dense open subscheme then
  (i) $\Leftrightarrow $ (ii)
  $\Rightarrow$ (iii) with
  \begin{itemize}
  \item[(i)] $ X'\to  X$ is curve tame;
    \item[(ii)] $ U'=X'\times_X U\to U$ is curve tame;
    \item[(iii)] the base change $X'_{\sO'}\to X_{\sO'}$ is curve tame.
  \end{itemize}
  If moreover $\sO'$ is a tame extension of $\sO$ then all conditions are equivalent.
\end{lem}

Assume moreover that $\overline X\setminus X$ is
a simple normal crossings divisor.
We say that $X'\to X$ is
{\em Grothendieck-Murre tame} if it is tame over all irreducible components of $\overline
X\setminus X$.

\begin{lem}\label{lem:tamegmcurve}
Under the above assumptions an \'etale covering $X'\to X$ is curve tame if and only if it
is Grothendieck-Murre tame.
\end{lem}

 A constructible sheaf $\sF$ on $X$ is called  tame  if for any closed point $x\in X$
  the $\mathrm{Gal}(\overline x /x)$-representation $\sF_{\overline x}$ is tame.

\subsection{Tameness of Picard-Lefschetz sheaves}

Consider a semistable Lefschetz
pencil of $X$ and let $\phi\colon \tilde X\to \P^1_\sO$ be the pencil map. Let   $U\subset
\P^1_K$  be the set of regular values over $K$.
Consider the
constructible sheaves $\mathsf
L^{i}=R^{i}\phi_{K,*} \Lambda $ on $\P^1_K$  for $i\in \Z$ and set $\sL = \oplus_i \sL^i$.

Note that $\mathsf L$ is tame over the prime divisor $\P^1_k$ of $\P^1_{\sO}$ by
Proposition~\ref{prop:grorzuni},
since $X$ has semistable reduction over the maximal point of $\P^1_k$.
We do not know whether $\mathsf L$ is tame over $\sO$ . The main result
of this section is the potential tameness of  $\mathsf L$.
In order to formulate it let $K'/K$ be a finite splitting field of all the field
extensions $k(x)/K$, for $x\in \mathbb P^1_K\setminus U$ and let $\sO'\subset K'$ be its
ring of integers. Recall that we can bound $K'$ by Theorem~\ref{thm:mainpropermorse}.
In the following $p=\mathrm{ch}(k)$.

\begin{thm}\label{tame:tamethm}
The sheaf $\mathsf L_{K'} $ on $\P^1 _{K'}$ is tame   if $p\ne 2$.
\end{thm}

From Theorem~\ref{tame:tamethm} and Lemma~\ref{lem:tamepullb} we deduce the following corollary.
In fact by Theorem~\ref{thm:mainpropermorse}(iii) we can choose $K'/K$ to be a tame extension under the assumption $p>n+1$.

\begin{cor}\label{cor:tamecor}
For $p>n+1$ or $p=0$ the sheaf $\mathsf L$ is tame.
\end{cor}

\begin{proof}[Proof of Theorem~\ref{tame:tamethm}]
  We will assume without loss of generality that $\Lambda=\Z/\ell \Z$ and that $\sO$ is
  strictly henselian.
  The major part of the proof consists in showing the following claim.

\begin{claim}\label{claim:redtameregval}
$\sL_{U_{K'}}$ is a tame local system.
\end{claim}

Assuming Claim~\ref{claim:redtameregval} let us prove Theorem~\ref{tame:tamethm}. The
sheaf $\sL^i $ is unramified over $\P^1_K$ for $i\notin \{n-1, n\}$ by
Theorem~\ref{thm:glcllefschetz}, so Lemma~\ref{lem:tamepullb} and
Claim~\ref{claim:redtameregval} imply that $\sL^i$ is tame in this case. With the notation
of~\ref{subsec:pltheocl} consider a critical value  $x\in \P^1_{K'}$. We have
the exact sequence \eq{eq:lefontcoh}{ 0\to H^{n-1}(\tilde X_{\overline x}) \to
  H^{n-1}(\tilde X_{\overline \eta_x}) \to V_{\overline x} \to H^{n}(\tilde X_{\overline
    x}) \to H^{n}(\tilde X_{\overline \eta_x}) \to 0 } where we omit the coefficients
$\Lambda$ for simplicity. We consider the action of
$G_x =\mathrm{Gal}( \overline \eta_x / \eta_x)$ on this sequence, where
  $\eta_x\in \Spec \sO^\hh_{\P^1_{K'} , x}$ is the generic point.  The order of
\[\im (G_x)\subset \mathrm{Aut}(H^{*}(\tilde X_{\overline \eta_x}))\] is coprime to $p$ by
Lemma~\ref{lem:tamegmcurve}, Claim~\ref{claim:redtameregval} and the Abhyankar
lemma~\cite[Expos\'e XIII, Proposition 5.2]{SGA7.1}  applied to a semistable model $Z$ of
$\P^1_{K'}$ over $\sO'$ as in the proof of Claim~\ref{claim:redtameregval} below. Here we
use the Abhyankar lemma in order to see that for the specialization  $v\in Z_k$ of $x$
\[
 \im(  \pi_1(\Spec( \sO^\hh_{Z,v}[1/\pi])\setminus S^{\rm st} , \overline \eta_x ) ) \subset
 \mathrm{Aut}(\sL_{\overline \eta_x})
\]
is prime to $p$, where $S^{\rm st}$ is as in the proof of Claim~\ref{claim:redtameregval}.
 The order of  $\im (G_x)\subset
\mathrm{Aut}(V_{\overline x})$ is coprime to $p$ by
Proposition~\ref{prop:picardlefschetzform}.

So from the exact sequence~\eqref{eq:lefontcoh} we deduce that the order of
$\im (G_x)\subset
\mathrm{Aut}( H^{i}(\tilde X_{\overline x}))$ is coprime to $p$ for $i\in\{n-1, n\}$. This
means that the
action of $\mathrm{Gal}(\overline x/x)$ on $ \sL^i_{\overline x}=H^{i}(\tilde X_{\overline x})$ is tame.
\end{proof}

\begin{proof}[Proof of Claim~\ref{claim:redtameregval}]
  By replacing $\sO'$ by a further tame extension we can in the following assume that the
 ramification index $e$ of $\sO'/\sO$ satisfies $e>n+1$, see Lemma~\ref{lem:tamepullb}.
 (We use this assumption in the proof of Lemma~\ref{lem:nearylogsm} below.)
 Let $\pi'\in
 \sO'$ be an uniformizer.
 Let $S\hookrightarrow \P^1_{\sO'}$ be the set of critical values.
 We construct a semistable model of
$\P^1_{K'}$ and check Grothendieck-Murre tameness over this model, which is sufficient by
Lemma~\ref{lem:tamegmcurve}.

 By our assumption on
 $\sO'$ we deduce that $S$ is the union of the images of finitely many sections of
 $\P^1_{\sO'}\to \Spec \sO'$. There is an iterated blow-up $\theta\colon Z\to \P^1_{\sO'}$ of closed points in the smooth locus
 over $\sO'$ such that the strict transform $S^{\rm st}\hookrightarrow Z$ of $S$ is a
 finite disjoint union of sections of $Z\to \Spec \sO'$; this is a special case of N\'eron
 desingularization, see~\cite[Corollary~4.6]{Art69}.
As in the process we only blow up smooth closed points the scheme $Z$ is semistable over
$\sO'$.

Let $z$ be a maximal point of $Z_k$  such that $\theta  (z)$ is a closed point of $\P^1_{\sO'}$.
Consider the henselian discrete valuation ring
$R=\sO^\hh_{Z,z}$.  Then $\theta_{z}\colon
\sO^\hh_{\P^1_{\sO'}, \theta( z )}\to R$ sends both
$\pi'$ and  $T$ to uniformizers, where $\pi',T \in \sO^\hh_{\P^1_{\sO'}, \theta( z )} $
is a regular parameter system. Say $T=-v \pi'$ with $v\in R^\times$.

Let $x\in X_{k} (\hookrightarrow X)$ be a non-critical point of $\phi_k$ and denote by abuse of notation its unique  preimage in $X_R $ under the
morphism $X_R\to X$
 by
the same symbol.  Then by
Proposition~\ref{prop:semistmorselem} we obtain an isomorphism of $R$-algebras
\ga{}{
\sO_{X_R,x}^\hh \cong R[X_0,\ldots ,
              X_n]^\hh/(X_0 \cdots X_m- u \pi, X_{m+1}+v \pi' ) \notag}
 which after a coordinate transformation  can be rewritten as
\ga{}{ \sO_{X_R,x}^\hh \cong R [X_0,\ldots ,
              X_{n-1}]^\hh /(X_0 \cdots X_m-  \pi'^e).\notag}
Here the henselization  $^\hh$  is with respect to the maximal ideal generated by  $\pi'$
and the $X_i$.

In particular $X_R$ is log smooth over $R$ around $x$ by Example~\ref{tame:exlogsm1}.
By a similar calculation one sees that at a critical point $x,$ the scheme $X_R$ is at
least nearly semistable over $R$ as defined in~\ref{subsec:nearlysemist} below.
By Proposition~\ref{prop.nearlysemist} applied over $R$ we see that $\mathsf L_{K'}$ is
tame over the divisor
$\overline{\{ z \}}$ for all $i\in \Z$.
 So $\sL_{K'}$ is Grothendieck-Murre tame over $\sO'$.
 \end{proof}

\subsection{Nearly semistable reduction}\label{subsec:nearlysemist}

In this subsection we do not assume that the residue field $k$ of $\sO$ is perfect, but we
assume that $k$ is separably closed for simplicity.
A scheme $X$ which is flat and of finite type over $\sO$ is said to have {\em nearly
semistable reduction} at $x\in X$ if either it is log smooth at $x$ over $\sO$  with the standard log structure $1\mapsto \pi$, a uniformizer of $\sO$,   or if there is an
isomorphism of $\sO$-algebras $\sO^\hh_{X,x}\cong A$. Here $A$ is of the following form:
consider an $\sO$-algebra
\[
  B=\sO[X_0, \ldots ,
   X_m]^\hh/(X_0 \cdots  X_m - u \pi^e)
 \]
 where $m\ge 0$, $u\in ( \sO[X_0,\ldots , X_m]^\hh)^\times$ and where  $e>m+1$.
 Here the henselization  $^\hh$  is with respect to the ideal $(\pi, X_0,\ldots, X_m)$.
Set $\alpha=X_0 + \ldots + X_m
+\pi\in B$. Then $A$ is of the form
\[
  A= B[X_{m+1} , \ldots , X_n]^\hh/(  X_{m+1}^2 + \ldots + X^2_n - \alpha)
\]
for some $n\ge m$.

\begin{prop} \label{prop.nearlysemist}
  Assume that $X$ is nearly semistable at all points of its closed fiber. Then the following properties are satisfied.
  \begin{itemize}
  \item[(i)]  If $X$ is proper over $\sO$ then the action of $\mathrm{Gal}(\bar K/ K)$ on $H^*(X_{\bar
  K}, \Lambda)$ is tame.
\item[(ii)] The action  of $\mathrm{Gal}(\bar K/ K)$ on $R\Uppsi_{X/\sO} (\Lambda)$ is tame.
\end{itemize}
 \end{prop}

 \begin{proof}
   Without loss of generality $\Lambda=\Z/\ell\Z$.
 Part (i) follows from part (ii) and proper base change.
 If $X$ is log smooth around  $x\in X_k$ then part (ii) follows from~\cite[Theorem~0.1]{Nak98}. So for
 part (ii) we have to
 consider a closed point $x\in X_k$ with $\sO^\hh_{X,x}\cong A$. Write $g\colon X_x\to Y_y$
 for $\Spec(A)\to \Spec (B)$, $y=g(x)$.
Observe that  the critical values of $g$  are given by $Y^{\rm crit} = V(\alpha)$. Set $\tilde Y = \mathrm{Bl}_{\{ y \}}(Y_y) \xrightarrow{\sigma} Y_y$ and $\tilde X = X_x\times_{Y_y} \tilde
Y\xrightarrow{\sigma} X_x$. Let $\tilde x\in \tilde X$ be a closed point over $x$ with image
$\tilde y\in \tilde Y$.

Consider  the
morphisms of henselian local schemes $\tilde g_{\tilde x}\colon \tilde X_{ \tilde x} \to \tilde Y_{\tilde y}$ and
$\tilde h_{\tilde y}\colon  \tilde Y_{\tilde y} \to \Spec\sO$.
We consider the object
\[
\mathsf F =\mathrm{cone} [ \Lambda_{ \tilde Y_{\tilde y}} \to R  \tilde g_{\tilde x, *}
\Lambda_{ \tilde X_{ \tilde x} } ]
\]
of $D^b_c(\tilde Y_{\tilde y},\Lambda)$ which is constructible
by~\cite[Corollaire~6.2]{Org06}  in which we do not have to perform the modification by~\cite[Proposition~4.1]{Org06}.
We need to show that \[  R\Uppsi_{\tilde X/\sO}(\Lambda)_{\tilde x} = [R \tilde h_{\tilde y,*}  R \tilde  g_{\tilde x, *}
\Lambda_{ \tilde X_{ \tilde x} }]_{\bar K} \] is tame, since then by proper base change
\[
 ( R\sigma_*  R\Uppsi_{\tilde X/\sO}(\Lambda))_{ x} \cong R\Uppsi_{ X/\sO}(\Lambda)_{x}
\]
is also tame.
By \cite[Theorem~0.1]{Nak98} and
Lemma~\ref{lem:nearylogsm} we obtain that \[
  R\Uppsi_{\tilde Y/\sO }(\Lambda)_{\tilde y}= R
  \tilde h_{\tilde y,*} (\Lambda )_{\bar K}
\]
is tame. So it remains to show that
$R\Uppsi_{\tilde Y/\sO }(\mathsf F)_{\tilde y}= R
\tilde h_{\tilde y,*} (\mathsf F )_{\bar K}$ is tame.

\begin{lem}
  $\sH^i(\mathsf F)$ vanishes for $i\ne n-m-1$ and
   $\sH^{n-m-1}(\mathsf F)=j_! \mathsf L$, where $j\colon \tilde Y_{\tilde y} \setminus
   \sigma^{-1} Y^{\rm crit} \to \tilde Y_{\tilde y}$. Here $\mathsf L$ is a $\Lambda$-rank
   one local system of order $1$ or $2$. In the latter case $\mathsf L$ becomes trivial after taking
   a square root of $\sigma^{-1}(\alpha)$.
 \end{lem}

 \begin{proof}
Combine~\cite[Expos\'e XV, Section 2.2]{SGA7.2} and  \cite[Proposition~4.1]{Org06}.
   \end{proof}

We can locally  factor  $\sigma^{-1}(\alpha)$  in the ring $\sO_{\tilde
  Y, \tilde y}$ as $\alpha^{\rm st} \beta$, where the vanishing locus of $\beta$ is
contained in the exceptional
divisor of the blow-up, i.e.\ $\beta$ is invertible over $K$ while
  $\alpha^{\rm st} $ does not vanish on the exceptional divisor.

\begin{lem}\label{lem:nearylogsm}
\'Etale locally on the exceptional divisor, $\tilde Y$ looks like
Example~\ref{tame:exlogsm1}
or Example~\ref{tame:exlogsm2} with $\alpha^{\rm st}$ invertible or corresponding to the variable
$X_{m+1}$. In particular, $\tilde Y_{\tilde y}$ and its closed subscheme $V(\alpha^{\rm
  st})$ are essentially  log smooth over $\sO$ and the immersion is strict.
\end{lem}

\begin{proof}
It suffices to consider without loss of generality the following two blow-up charts of
$\tilde Y$ over $Y$.  For simplicity of notation we assume that $k$ is algebraically closed.

\smallskip

{\em 1st chart.} \\
$\tilde Y$ contains as an open subscheme the spectrum of the ring
\[
B[\tilde X_0 ,\ldots , \tilde X_m] /  (\tilde X_i \pi -X_i , \tilde X_0 \cdots\tilde X_m -u \pi^{e-m-1})
\]
where $1\le i\le m$. On this chart $\alpha^{\rm st}= \tilde X_0 + \ldots + \tilde X_m +1$.
Consider a closed point $\tilde y$ of the exceptional
divisor in
this chart which satisfies without loss of generality \[\tilde X_0(\tilde y)= \cdots =
\tilde X_{\tilde m}(\tilde y)=0 \text{ and } \tilde X_{\tilde m+1}(\tilde y), \ldots , \tilde
X_m(\tilde y)\ne 0.\] We obtain
\[
\sO^\hh_{\tilde Y,\tilde y} \cong  \sO[X_0 , \ldots ,X_{m} ]^\hh / (  X_0 \cdots X_{\tilde m}
- \pi^{e-m-1})
\]
in which $\alpha^{\rm st}$ is a unit or $\alpha^{\rm st}=X_{\tilde m+1}$. In this case
$\sO^\hh_{\tilde Y,\tilde y}$ is as in Example~\ref{tame:exlogsm1}.

\smallskip

{\em 2nd chart.} \\
$\tilde Y$ contains as an open subscheme the spectrum of the ring
\[
B[\tilde X_1, \ldots , \tilde X_m,\tilde \pi]/(\tilde X_i X_0 -X_i , \tilde \pi X_0- \pi ,
\tilde X_1 \cdots \tilde X_m - u \tilde \pi^{m+1}  \pi^{e-m-1} )
\]
where $1\le i\le m$. On this chart $\alpha^{\rm st}=  1+ \tilde X_1 + \ldots + \tilde X_m
+ \tilde \pi $. Consider a closed point $\tilde y$ of the exceptional divisor in this
chart. We can assume without loss of generality $\tilde \pi(\tilde y)=0$ as other points
are already in the 1st chart above. Say for simplicity that $\tilde X_1(\tilde y) = \cdots = \tilde X_{\tilde
  m}(\tilde y)=0$ and $\tilde X_{\tilde m+1}(\tilde y), \ldots ,  \tilde X_{m}(\tilde
y)\ne 0$. Then we get
\[
\sO^\hh_{\tilde Y,\tilde y} \cong  \sO[X_0 , \ldots ,X_{m},\tilde \pi ]^\hh / (  \tilde
\pi X_0 - \pi,    X_1 \cdots X_{\tilde m} - \tilde \pi^{m+1} X_0^{e-m-1}).
\]
in which $\alpha^{\rm st}$ is a unit or $\alpha^{\rm st}= X_{\tilde m +1}$.
In this case
$\sO^\hh_{\tilde Y,\tilde y}$ is as in Example~\ref{tame:exlogsm2}.
\end{proof}

We resume  the proof of Proposition~\ref{prop.nearlysemist}. We argue case by case.

 {\em 1st case: $\mathsf L$ is constant.}\\
 Consider the exact sequence of sheaves on $ \tilde Y_{\tilde y}$
 \[
0\to j_! j^*\Lambda  \to \Lambda \to i_* i^*  \Lambda \to 0
\]
where $i\colon  \sigma^{-1}(Y^{\rm crit}) \to \tilde Y_{\tilde y}$ is the closed immersion
and $j$ the complementary open immersion.
The corresponding long exact sequence for $ R\tilde h_{\tilde y,*}  ( - )_{\bar K}$, \cite[Theorem~0.1]{Nak98}
and Lemma~\ref{lem:nearylogsm} imply the requested tameness for $  R\tilde
h_{\tilde y,*}  (  j_!\mathsf L )_{\bar K}$.

\medskip

 {\em 2nd case: $\mathsf L$ is non-constant (in particular $\ell\ne 2$).}\\
 Let $ \hat Y_{\tilde y}\to \tilde Y_{\tilde y}$ be the Kummer covering corresponding to
 adjoining the square root of $\alpha^{\rm st}$. Then by  Lemma~\ref{lem:nearylogsm}
 $\hat Y_{\tilde y}$ and its closed subscheme $V(\alpha^{\rm st})$ are
   essentially  log smooth over
 $\sO$ and the immersion is strict,    as Kummer coverings are log smooth.

 Consider the morphism of log schemes
\[\hat Y_{\tilde y}\to \Spec \Z[T], \quad T\mapsto \beta .
  \]
  Consider the Kummer log \'etale covering
  \[
\Spec  \Z[T'] \to  \Spec  \Z[T],\quad T\mapsto (T')^2
\]
and the corresponding Kummer log \'etale covering of fs log schemes \[    Y_{\tilde y}^{\mathsf L}:= \hat Y_{\tilde y} \otimes^{\rm
  fs}_{\sO[T]} \sO[T'] \to \hat Y_{\tilde y} . \]
Note that this fs log base change agrees with the ordinary base change over $K$.
Then $ Y_{\tilde y}^{\mathsf L}$ and its closed subscheme $V(\alpha^{\rm st})$ are log
  smooth over $\sO$.

Consider the finite morphism
\[
  \mu\colon  Y_{\tilde y}^{\mathsf L} \to \tilde Y_{\tilde y}
\]
of degree $4$.
Then the pullback of the local system $\mathsf L$ along $\mu$ is trivial and $\sF$ is a direct summand of
$R\mu_* \mu^* \sF $ as $\ell \ne 2$ and as $\sF$ vanishes on the ramification locus of
$\mu$. We can now argue for $\mu^*(\mathsf F)$ on $ Y_{\tilde y}^{\mathsf L}$ as we
did in the first case for $\sF$ on
$\tilde Y_{\tilde y}$.
\end{proof}

\appendix

\section{Nearby cycle functor}\label{sec:appendix}

In this appendix we describe basic properties of the \'etale nearby cycle functor in a
coordinate free fashion using so called Iwasawa twists.  We hope that this coordinate free
presentation makes the Verdier duality theory   and the discussion  of Rapoport-Zink sheaves in Section~\ref{sec:rapzi} more
transparent. Nothing we present was not known to people in the early 1980s. We  recast the theory
using the pro-\'etale topology.

\subsection{Reminder on constructible sheaves}\label{sec:constrsh}
Let $Y$ be a noetherian scheme. For any ring $\Lambda,$ let $\mathrm{Sh}(Y,\Lambda)$ be the
category of pro-\'etale sheaves of $\Lambda$-modules and $D(Y,\Lambda)$ be  its derived
category, see~\cite{BS15}  or  \cite{Ker16}. Assume in the following that $\Lambda$ is a
complete local noetherian ring with maximal ideal $\mathfrak m$. We
define a {\it  constructible $\Lambda$-sheaf} $\sF$ as a pro-\'etale sheaf of $\Lambda$-modules
such that
$\sF = \lim_n \sF /\mathfrak m^n \sF $ and such
that for all $n>0$ the pro-\'etale sheaf of $\Lambda/\mathfrak m^n$-modules
$\sF /\mathfrak m^n \sF$ comes from an \'etale constructible
sheaf of $\Lambda/\mathfrak m^n$-modules, see~\cite[09BS]{StPr}.
Let $\mathrm{Sh}_c(Y,\Lambda)$ be the category of constructible $\Lambda$-sheaves.

One can show~\cite{Kri23}
that $\mathrm{Sh}_c(Y,\Lambda)$  forms a noetherian abelian subcategory closed under
extensions of the category of all pro-\'etale sheaves of $\Lambda$-modules. One also
shows~\cite[09BS]{StPr}
that for $\sF\in \mathrm{Sh}_c(Y,\Lambda)$  there exists a
stratification $\mathbf Z$ of $X$ such that $\sF|_Z$ is smooth for all $Z\in \mathbf Z$.

Let $\Lambda_\circ$ be a localization of $\Lambda$. We set
$\mathrm{Sh}_c(Y,\Lambda_\circ)= \mathrm{Sh}_c(Y,\Lambda) \otimes_\Lambda \Lambda_\circ$.
Let $t$ be an endomorphism of $\sF \in \mathrm{Sh}_c(Y,\Lambda_\circ)$ and assume that
$\Lambda_\circ$ is artinian. Then there exists a unique decomposition
$\sF= \sF^{\rm nil} \oplus \sF^{\rm inv}$ stable under $t$ such that $t$ is nilpotent on
$\sF^{\rm nil} $ and invertible on $\sF^{\rm inv}$. Uniqueness is clear while for
existence we define
$\sF^{\rm nil}=  \ker (t^n \colon \sF\to \sF)$ and
$\sF^{\rm inv}= \im (t^n \colon \sF\to \sF)$ for $n\gg 0$.  One can check  the property on
a stratification as above on which it reduces to the case of finite
$\Lambda_\circ$-modules
where it is the classical
Weyr-Fitting decomposition~\cite[Proposition 2.2]{Bou22}.

\smallskip

Let $D_c(Y,\Lambda)$ be the triangulated subcategory of complexes with constructible cohomology sheaves inside the derived category of pro-\'etale sheaves of
$\Lambda$-modules $D(Y,\Lambda)$. Let $D_c^b(Y,\Lambda)$ be the triangulated
subcategory which additionally has bounded cohomology sheaves. Define
$D_c^b(Y,\Lambda_\circ)$ as the Verdier localization ``up to isogeny''
$D_c^b(Y,\Lambda) \otimes_\Lambda \Lambda_\circ$, for a ring $\Lambda_\circ$ which  is a localization
of $\Lambda$.

For a morphism of noetherian schemes $f\colon Y_1\to Y_2$ there are natural functors
\[f^*\colon \mathrm{Sh}_c(Y_2,\Lambda)\to  \mathrm{Sh}_c(Y_1,\Lambda)\] and $f^*\colon
D^b_c(Y_2,\Lambda) \to D^b_c(Y_1,\Lambda) $.
From this $f^*$ one   deduces the derived pushforwards \[Rf_!,Rf_* \colon
D^b_c(Y_1,\Lambda) \to  D^b_c(Y_2,\Lambda)\] by the usual
adjunctions whenever they have a chance to exist~\cite[Introduction, Th\'eor\`eme~1]{ILO14}. One also gets the
exceptional pullback $f^!\colon D^b_c(Y_2,\Lambda)  \to D^b_c(Y_1,\Lambda) $ by adjunction.

If $\Lambda$ is more generally  a filtered colimit  $\Lambda=\mathrm{colim}_{j} \Lambda_j$
of complete local noetherian rings $\Lambda_j$ with flat, finite transition homomorphisms we set
\[
D_c(X,\Lambda)=\mathop{\rm colim}_j D_c(X,\Lambda_j)
\]
and similarly for a localization $\Lambda_\circ$ of $\Lambda$. This might depend on the
system of the $\Lambda_j$.

\subsection{Iwasawa twists}\label{subsec_iwasawa}

In this subsection we collect some results from~\cite{Bei87},  \cite{LZ19}.

Let $G$ be a profinite group which is isomorphic to $\Z_\ell$. We do not fix an isomorphism.
Let $\Lambda$ be a noetherian complete local ring with residue characteristic $\ell$.   Consider the ``Iwasawa algebra''
$\Lambda^\Iw=\Lambda\llbracket G \rrbracket$. The augmentation ideal
$\mathfrak I =\mathrm{ker} (\Lambda^\Iw\to \Lambda)$ is generated by a non-zero divisor $[\xi]-1$, where
$\xi\in G$ is a topological generator. Indeed,  such a choice induces an isomorphism
\[
  \Lambda \llbracket t \rrbracket \xrightarrow{\sim}\ \Lambda^\Iw ,\quad t\mapsto [\xi ]-1.
\]

Let $Y$ be a noetherian scheme.
Consider the derived category of sheaves of torsion Iwasawa modules ``with vanishing $\mu$-invariant'' up to isogeny
\[
  D^{\Iw }(Y,\Lambda_\circ) :=[ D(Y,\Lambda^\Iw)\cap R^{-1} ( D^b_c(Y,\Lambda))]\otimes_\Lambda \Lambda_\circ.
\]
Here $R\colon D(Y,\Lambda^\Iw) \to D(Y,\Lambda)$ is induced by the homomorphism
$\Lambda\to \Lambda^\Iw$.
Let $D^{\rm nil}(Y,\Lambda_\circ)$ resp.\ $D^{\rm inv}(Y,\Lambda_\circ)$ be the full
subcategory of $ D^{\Iw }(Y,\Lambda_\circ)$ on which $t$ is nilpotent resp.\ invertible.
 For the rest of this subsection we assume that $\Lambda_\circ$ is artinian.

\begin{lem} \label{lemapp:nilinvdec}
 We have the decomposition
  \[ D^{\Iw }(Y,\Lambda_\circ) = D^{\rm nil}(Y,\Lambda_\circ) \oplus D^{\rm inv}(Y,\Lambda_\circ)
    \]
\end{lem}

\begin{proof}
The two subcategories $D^{\rm nil}(Y,\Lambda_\circ)$ and $D^{\rm inv}(Y,\Lambda_\circ)$ of $ D^{\Iw }(Y,\Lambda_\circ) $ are orthogonal, as in the classical Weyr-Fitting decomposition \cite[Proposition~2.2]{Bou22}. In particular, if an object has a decomposition, then it is unique.
In order to show that an object decomposes,
 we use the exact
triangle
\[
\tau^{<a} \sF \to \tau^{\le a } \sF \to \mathcal H^a(\sF)[-a]\to \tau^{<a} \sF[1],
\]
induction on $a$ and the above decomposition
\[
  \mathcal H^a(\sF) = \mathcal H^a(\sF)^{\rm nil}
\oplus \mathcal H^a(\sF)^{\rm inv}.
\]
In fact this induces the required decomposition
\[
 \tau^{\le a } \sF \cong \mathrm{hofib}(   \mathcal H^a(\sF)^{\rm nil}[-a] \to  \tau^{<a}
 \sF^{\rm nil}[1]  ) \oplus  \mathrm{hofib}(   \mathcal H^a(\sF)^{\rm inv}[-a] \to  \tau^{<a}
 \sF^{\rm inv}[1]  ).
\]
\end{proof}

If we work with an arbitrary profinite group $H$  in place of $G \cong \Z_\ell$,  we have to  slightly reformulate.
Let $\mathrm{Sh}_c(Y \times B H,\Lambda)$, $D^b_c(Y\times B H,\Lambda)$ etc.\ be the $H$-equivariant
versions of the definitions from~\ref{sec:constrsh}, for example for $\sF \in
\mathrm{Sh}_c(Y,\Lambda)$ the $H$-action factors through a finite discrete quotient of $H$
on $\sF/\mathfrak m^n$ for all $n>0$.
We obviously have an equivalence $D^b_c(Y \times BG,\Lambda_\circ)= D^\Iw(Y,\Lambda_\circ)$.

Consider an epimorphism $H\to G$ of pro-finite groups whose kernel $W$ has pro-order coprime
to $\ell$ then  we denote by
\[
\mathrm{Nil}\colon D^b_c(Y\times BH,\Lambda_\circ) \to D^\nil(Y,\Lambda_\circ),\quad \mathsf F
\mapsto (\mathsf F^W)^{\rm nil}
\]
 the projection to the nilpotent part of the $W$-invariants, see Lemma~\ref{lemapp:nilinvdec}.

For $\mathsf F \in D^\Iw(Y,\Lambda_\circ)$ and $a\in \Z$ we call
\[
  \mathsf F (a)^\Iw = \mathsf F \otimes_{\Lambda^\Iw} \mathfrak I^a
\]
the {\em Iwasawa twist} of $\mathsf F$. Note that $\mathfrak I\subset \Lambda^\Iw$ is an
invertible ideal.
Clearly, if $G$ acts trivially on $\sF$ there is a canonical isomorphism between
$\sF(a)^\Iw$ and the {\em Tate twist}
\[
   \sF(a) =  \sF\otimes_{\Z_\ell} G^{\otimes a},
\]
since  $\mathfrak I^a/\mathfrak
I^{a+1}\cong \Lambda\otimes_{\mathbb Z_\ell}  G^{\otimes a}$.

\smallskip

For $\sF\in D^\Iw(Y,\Lambda_\circ)$ let $\sF^G\in D^b_c(Y,\Lambda_\circ)$ be the derived
$G$-invariants. The canonical action map $\sF \otimes_{\Lambda^\Iw} \mathfrak I\to \sF$
can be written as $N^\Iw\colon \sF\to \sF(-1)^\Iw$. Then we obtain the fundamental exact triangle
\begin{equation}\label{eq:app1:fundext}
\sF^G  \to \sF \xrightarrow{N^\Iw} \sF(-1)^\Iw \to \sF^G[1]
\end{equation}
For the derived $G$-coinvariants $\sF_G$ we have a canonical isomorphism  $ \sF_G
\cong\sF^G(1)[1]$.

Any $\sF\in D^{\rm nil}(Y,\Lambda_\circ)$ is derived $t$-complete and
there is an isomorphism $\sF\otimes_{\Lambda^\Iw} \Lambda= \sF_G$. Together with the
derived Nakayama Lemma~\cite[ Lemma~0G1U]{StPr}  we obtain the following lemma.

\begin{lem}\label{app:lem:derivednaka}
A morphism $\sF\to \sG$ in $D^\nil(Y,\Lambda_\circ)$ is an isomorphism if $\sF^G \to \sG^G$ is an isomorphism.
\end{lem}

The group inverse $G\to G, g\mapsto -g$ induces an involutive ring isomorphism
$\mathrm{inv}\colon \Lambda^\Iw\to \Lambda^\Iw$ preserving $\mathfrak I$.
We denote the
sheaf $\sF\in D^\Iw(Y,\Lambda)$ with the $\mathrm{inv}$-twisted action by $F^-$.
The involution applied to the Iwasawa twist induces an isomorphism in $D^\Iw(Y,\Lambda)$
\begin{equation}\label{eq:app1_invo}
\mathrm{inv}(a) = \mathrm{id}_\sF\otimes \mathrm{inv}^{\otimes a}\colon (\sF(a)^\Iw)^- \xrightarrow\sim (\sF^-)(a)^\Iw.
\end{equation}

 Let $\varpi\in D^b_c(Y,\Lambda)$ be a dualizing sheaf and let
\[
\mathrm D (\sF) = \mathit{Hom}_{\Lambda}(\sF^- , \varpi)
\]
be the Verdier dual of $\sF\in D^\Iw(Y,\Lambda)$, where the action of $\Lambda^\Iw$ on $\varpi$ is trivial.

We obtain a commutative diagram, which clarifies in which sense $N^\Iw$ is compatible with
duality
\begin{equation}\label{eq:app1:compiwdu}
  \xymatrix{
    \mathrm D (F(-1)^\Iw )  \ar[r]^{\mathrm D(N^\Iw)} \ar[r]  & \mathrm D(F) \\
    \mathit{Hom}(F^-(-1)^\Iw , \varpi ) \ar@{=}[r] \ar[u]^{\mathrm{inv}(-1)} &  \mathrm D (F)(1)^\Iw\ar[u]_{N^\Iw (1)^\Iw}
  }
\end{equation}
Note that in this duality statement the strength of Iwasawa twists becomes obvious, as previous attempts to write
down this compatibility looked quite ad hoc, see Deligne's letter to MacPherson~\cite{Del82}.

\smallskip

 For $\Q\subset \Lambda_\circ$ and $\sF \in D^\nil(Y,\Lambda_\circ )$  there is a canonical isomorphism
 between Tate and Iwasawa twists $c_{\rm TIw}\colon \sF(a)\xrightarrow{\sim}
 \sF(a)^\Iw$ for $a\in \Z$, which for $a=1$ is induced by the homomorphism
 \[
G\to ( \mathfrak I/\mathfrak I^\nu) \otimes_{\Z} \Q,\quad g \mapsto \log ([g]),
\]
where $\nu >0$ is chosen such that $\mathfrak I^\nu$ acts trivially on $\sF$.
This gives us a canonical nilpotent map $N\colon \sF\to \sF(-1)$ such
that
 \[
   \xymatrix{
     & \sF(-1)  \ar[dd]_{\wr}^{c_{\rm TIw}} \\
   \sF \ar[ru]^-{N} \ar[rd]_-{N^\Iw}  & \\
  & \sF(-1)^\Iw
   }
 \]
commutes.
This means that
$g\in G$ acts by $\exp(g N)\colon \sF\to \sF$, where $gN\colon \sF\to \sF$ is the
contraction of the Tate twist. So we see that in this case the fundamental exact triangle reads
\[
\sF^G \to \sF \xrightarrow{N} \sF(-1) \to  \sF^G[1] .
\]
Via the identification $c_{\rm TIw}$ the isomorphism~\eqref{eq:app1_invo} becomes
multiplication by $(-1)^a$, so by abuse of notation we can write the commutativity of~\eqref{eq:app1:compiwdu} simply as $\mathrm D
(N)=-N(1)$.

\subsection{Unipotent nearby cycles}\label{app_subsec_unip}

Let $\sO$ be a strictly henselian discrete valuation  ring  such that the prime number $\ell$ is
invertible in $\sO$. Let $K$ be the fraction field of $\sO$ and $k$ be the residue field.
In this appendix we do not need to assume that $k$ is perfect or that $\mathrm{ch}(K)=0$. Let
$\overline K$ be a separable algebraic closure of $K$. For
simplicity of notation we just write $\Lambda$ for the coefficient ring, which could be a
localization of a noetherian complete local ring or a filtered colimit of such as in~\ref{sec:constrsh}, and we drop the
notation $\Lambda_\circ$ used above.

Let $f\colon X\to \Spec\sO$ be a
scheme of finite type. Set $H=\mathrm{Gal}(\overline K /K)$ and let   $H\to
\Z_\ell(1)_{\overline k}$  be the  tame character of $H$. We set $G=\Z_\ell(1)_{\overline k}$.
In
the following we use the notation of~\ref{subsec_iwasawa}.
Let $j\colon X_K\to X$ and $i\colon X_k\to X$ be the immersions of fibers and consider the
morphism of topoi
$\bar\jmath\colon X_{\overline K}\to X \times BH$ and $\bar\jmath_K \colon X_{\overline K}\to X_K$.

Recall~\cite[Expos\'e XIII]{SGA7.2} that the {\em nearby cycle functor} is defined as
\[
R\Uppsi_{X/\sO}\colon D(X_K,\Lambda)\to D(X_k\times BH,\Lambda),\quad \sF\mapsto  i^*
R\bar\jmath_* \bar\jmath^*_K (\sF).
\]
Deligne showed~\cite[Chapitre~7]{SGA4.5} that $R\Uppsi_{X/\sO}$ preserves bounded constructible
complexes. Note that we have the identity
\begin{equation}\label{app:lnearbyc}
  (R\Uppsi_{X/\sO})^W = i^* R \tilde\jmath_*  \tilde\jmath^*_K \colon D(X_K,\Lambda) \to D^\Iw(X_k,\Lambda),
\end{equation}
with the morphisms of topoi $\tilde\jmath \colon X_{\overline K^W}\to X\times BG $ and
$\tilde\jmath_K \colon X_{\overline K^W}\to X_K$, which is useful for the general residue
field case in~\ref{subsec:appnonsepcl}.

From now on we assume that $\Lambda$ is artinian.
We
write
\[
\uppsi= \uppsi_{X/S}= \mathrm{Nil}\, R\Uppsi_{X/\sO}[-1]\colon D^b_c(X_K, \Lambda)\to D^\nil(X_k,\Lambda)
\]
for the shifted unipotent nearby cycle functor, see Lemma~\ref{lemapp:nilinvdec}.

The fundamental exact triangle~\eqref{eq:app1:fundext} applied to $\uppsi(\sF)$ now reads
\begin{equation}\label{eq_app1:fundncf}
\uppsi(\sF)\xrightarrow{N^\Iw} \uppsi(\sF)(-1)^\Iw \to i^* j_* \sF \to \uppsi(\sF)[1]
\end{equation}
as $[R\bar\jmath_* \bar\jmath^*(\sF)]^H=j_* (\sF)$.
Here we omit the right derived sign for
$j_*$  as $j_*$ is perverse t-exact for the t-structure in~\ref{app:subsecperv}.
This fundamental triangle gives a d\'evissage for
$\uppsi(\sF)$ in terms of $i^* j_*(\sF)$. A further d\'evissage of the latter is
accomplished by the exact triangle
\begin{equation}\label{eq_app1:fundncf1}
i^* \sG \to i^*j_* j^* \sG \to i^! \sG [1] \to i^*\sG[1]
\end{equation}
for $\sG\in D^b_c(X,\Lambda)$.

\subsection{Perverse sheaves}\label{app:subsecperv}

Assume in the following that $\Lambda=\Z/\ell^\nu \Z$ or that $\Lambda$  is an algebraic
field extension of $\Q_\ell$.
For $f\colon X\to \sO$ separated and of finite type \[\varpi= f^! \Lambda(1)[2] \in D^b_c(X,\Lambda) \] is a
dualizing sheaf, i.e.\ for $\sF\in D^b_c(X,\Lambda)$ and $\mathrm D(\sF):=
R\mathit{Hom}_X(\sF,\varpi)\in D^b_c(X,\Lambda )$ the canonical map $\sF\xrightarrow{\sim }  \mathrm D
(\mathrm D (\sF))$ is an isomorphism.

We always work with the (middle) perverse t-structure on $D^b_c(X,\Lambda)$ induced by the
dimension function \[ \delta_X(x)=\mathrm{trdeg}(k(x)/k(f(x))) + \dim_{\Spec \sO}(\overline{ \{ f(x)
    \}})
\]
as in~\cite{Gab04}. This means that
\[ \sF \in \per D^{\le 0}_c(X,\Lambda)\; \Leftrightarrow\;
  \mathcal H^i( i^*_{\overline x} \sF)= 0\]
for all $x\in X$ and $i>\delta_X(x)$ and $\sF
\in \per D^{\ge 0}_c(X,\Lambda) \Leftrightarrow  \mathrm D(F) \in   \per D^{\le 0}_c(X,\Lambda)$.

Concretely, this means that for $X_K$  regular, connected of dimension $n$ the
sheaf \[\Lambda[n+1]\in D^b_c(X_K,\Lambda)\] is perverse and for $X_k$ regular connected of
dimension  $n$ the sheaf $\Lambda[n]\in D^b_c(X_k,\Lambda)$ is  perverse.

Gabber showed, see also
\cite{Bei87}, \cite{Del82}, that $R \Uppsi [-1]$ and therefore $\uppsi$  are perverse t-exact with respect to this
perverse t-structure, see
\cite[Corollaire~4.5]{Ill94}.

\begin{rmk}\label{rmk:pervcontsp}
If $\sG\in D^b_c(X,\Lambda)$ is perverse with no non-trivial
subobjects or quotient objects supported on $X_k$, i.e.\ $\sG=j_{!*}j^* \sG $, then
\[ i^*\sG[-1]={}^p\sH^{-1}(i^* j_* j^*\sG) \] and $i^!\sG[1] = {}^p\sH^{0}(i^* j_* j^* \sG)$
are the (shifted) perverse constituents of $i^* j_* j^*\sG\in {}^p \!
D^{[-1,0]}_c(X_s,\Lambda)$ and~\eqref{eq_app1:fundncf1} is the corresponding truncation
exact triangle.
\end{rmk}

\subsection{Base change}

Let the notation be as in~\ref{app_subsec_unip} and let
$f\colon Y\to X$ be a morphism of finite type.

For $f$ proper and $\mathsf K$ in $D^b_c(Y_\eta)$. The canonical proper base change map
\[
\uppsi\circ f_*  \xrightarrow{\sim} f_*\circ \uppsi
\]
is an equivalence of functors from $D^b_c(Y_\eta,\Lambda)$ to $D^{\rm nil}(X_s,\Lambda)$.

There exists also a canonical base change map
\[
f^* \circ  \uppsi \to \uppsi \circ f^*
\]
of functors from $D^b_c(X_\eta,\Lambda)$ to $D^{\rm nil}(Y_s,\Lambda)$ which is an equivalence
if $f$ is smooth.
Note that we get an induced morphism of
exact triangles~\eqref{eq_app1:fundncf}
\begin{equation}\label{app:funceq1}
  \begin{gathered}
    \xymatrix{
f^* \uppsi(\sF)\ar[r]^-{N^\Iw} \ar[d] & f^* \uppsi(\sF)(-1)^\Iw \ar[r]  \ar[d] & f^* i^* j_*
\sF \ar[r]  \ar[d] &
f^* \uppsi(\sF)[1]  \ar[d]\\
\uppsi(f^*\sF)\ar[r]^-{N^\Iw}  & \uppsi(f^*\sF)(-1)^\Iw \ar[r] &  i^* j_* f^*\sF \ar[r] &
\uppsi(f^* \sF)[1]
}
 \end{gathered}
\end{equation}
for $\sF\in D^b_c(X_\eta,\Lambda)$. Similarly, we have a morphism of exact triangles~\eqref{eq_app1:fundncf1}
\begin{equation}\label{app:funceq2}
  \begin{gathered}
    \xymatrix{
f^* i^* \sG \ar[r] \ar[d] & f^* i^*j_* j^* \sG \ar[r] \ar[d] & f^*  i^! \sG [1] \ar[r]
\ar[d] & f^*  i^*\sG[1]  \ar[d]\\
i^* f^* \sG \ar[r] & i^*j_* j^*f^*  \sG \ar[r] &  i^!f^*  \sG [1] \ar[r] & i^*f^* \sG[1]
}
 \end{gathered}
\end{equation}
for $\sG \in D^b_c(X,\Lambda)$.

\subsection{Compatibility with Verdier duality}\label{app:verdierdu}

Write $\varpi_{X_s}=f^!_s (\Lambda)$, $\varpi_X=f^!(\Lambda(1)[2])$ and $\varpi_{X_\eta}=\varpi_X|_{X_\eta}$.
Consider $\sF,\sF'\in D^b_c(X_\eta,\Lambda)$ and assume they are endowed with  a pairing
\[
\mathsf{p}\colon \sF \otimes_\Lambda \sF' \to \varpi_{X_\eta}.
\]
As $\uppsi$ is lax symmetric monoidal up to a shift,  the composition of
\ml{}{
\uppsi(\sF)\otimes \uppsi(\sF') \to \uppsi(\sF\otimes \sF')[-1] \xrightarrow p \uppsi(f^!\Lambda(1)[1]) \to
  \\
   f_s^! \uppsi(\Lambda(1)[1]) \cong f_s^! \Lambda(1) = \varpi_{X_s}(1)
}
induces a pairing
\[
  \uppsi(\mathsf p)\colon \uppsi(\sF) \otimes_{\Lambda^\Iw} \uppsi(\sF')^- \to \varpi_{X_s}(1)
\]
in $D^\Iw(X,\Lambda)$, where $G$ acts trivially on the codomain.

By an analogous construction one gets a pairing
\begin{equation}\label{app:eq:pair2}
i^* j_*  \sF  \otimes_{\Lambda} i^* j_*  \sF'  \to \varpi_{X_s}[1].
\end{equation}

Via these pairing the fundamental exact triangle~\eqref{eq_app1:fundncf} becomes self-dual
in the sense that we get a commutative diagram
\begin{equation}\label{app:eq1comdual}
  \begin{gathered}
 \small \xymatrix{
\uppsi(\sF)\ar[r]^-{N^\Iw} \ar[d]  &  \uppsi(\sF)(-1)^\Iw \ar[d]\ar[r]  &  i^* j_* \sF
\ar[d] \ar[r] &  \uppsi(\sF)[1] \ar[d]   \\
\mathrm D(\uppsi(\sF'))(1)^\Iw \ar[r]^-{\mathrm D(N^\Iw)} & \mathrm  D( \uppsi(  \sF'))    \ar[r] & \mathrm  D(i^* j_* \sF')[1] \ar[r] &
\mathrm  D(\uppsi(\sF'))(1)^\Iw[1]
}
\end{gathered}
\end{equation}
in $D^b_c(X_s,\Lambda)$.

We say that $\mathsf p$ is a perfect pairing if it induces
an isomorphism $\sF \xrightarrow\sim \mathrm D (\sF')$.
Gabber~\cite[Th\'eor\`eme~4.2]{Ill94} showed the following property.

\begin{lem}
   If $\mathsf p$ is a perfect pairing then
  \begin{itemize}
  \item[(i)] the pairing $\uppsi(\mathsf p)$ is perfect and
\item[(ii)] the pairing \eqref{app:eq:pair2} is perfect.
\end{itemize}
\end{lem}

\subsection{Non-separably closed residue field}\label{subsec:appnonsepcl}

If the residue field $k$ of the henselian discrete valuation ring $\sO$ is not assumed to
be separably closed we proceed as follows.  For $k$ finite the theory is laid out in
 detail in~\cite{HZ23},  see also~\cite[Expos\'e XIII]{SGA7.2} for background on the
Deligne topos.

Let $\tilde K\subset \overline K$ be the subfield generated by all $\ell^n$-th roots of unity and  by all  $ \pi^{1/\ell^n}$ for $n>0$. Let $\tilde k$ be the residue field of
$\tilde K$.
Fix a splitting $\sigma$ of the exact sequence of profinite groups
\[
 0\to  G= \Z_\ell(1)  \to  \tilde G \to \tilde g \to 0
\]
where $ \tilde G = \mathrm{Gal}(\tilde K / K ) $, $\tilde g=  \mathrm{Gal}(\tilde k/k)$.
Such a splitting is for example induced by a choice of compatible $\ell^n$-th roots of
$\pi$ for $n>0$.
   Consider the morphisms of topoi $\tilde \jmath\colon X_{\tilde K} \to X\times_{B\tilde
     g} B\tilde G$ and $\tilde \jmath_K\colon X_{\tilde K}  \to X_K  $. The
   splitting $\sigma $ induces an isomorphism
   \eq{eq:idnonsepclres}{
     D^\Iw(X_k,\Lambda )  \cong
     D^b_c(X_k\times_{B\tilde g} B\tilde G , \Lambda ) ,
   }
   where Iwasawa module sheaves on
   $X_k$ are defined as before with $\Lambda^\Iw$ now  a ring in the category of pro-\'etale
   sheaves on $X_k$.
We define the unipotent nearby cycle functor
\[
\uppsi=   \mathrm{Nil}\,  i^* \tilde\jmath_* \tilde\jmath^*_K [-1]\colon D^b_c(X_K,\Lambda)\to
D^{\rm nil}(X_k,\Lambda)
\]
using the identification~\eqref{eq:idnonsepclres}.

This construction clearly depends  on the splitting $\sigma$.
However, for $\Lambda=\Z/\ell^\nu\Z$ or for $\Lambda $  an algebraic extension of $ \Q_\ell$ and for $\sF\in
D^b_c(X_K,\Lambda)$ perverse, the monodromy graded pieces $\gr^{\rm M}_a \uppsi (\sF)$, see Section~\ref{subsec:mofil}, do
not depend on the splitting $\sigma$ in view of the proof of~\cite[Proposition 5.1.2]{BBD83}.

\end{document}